\documentclass[11pt,a4paper]{article}
\usepackage[T1]{fontenc}
\usepackage{lmodern}
\usepackage[margin=26mm]{geometry}
\usepackage{amsmath,amssymb,amsthm,mathtools}
\usepackage{mathrsfs}
\usepackage{microtype}
\usepackage{enumitem}
\usepackage[hidelinks]{hyperref}
\usepackage{bookmark}
\usepackage{needspace}
\numberwithin{equation}{section}
\newtheorem{theorem}{Theorem}[section]
\newtheorem{proposition}[theorem]{Proposition}
\newtheorem{lemma}[theorem]{Lemma}
\newtheorem{corollary}[theorem]{Corollary}
\theoremstyle{remark}

\newcommand{\R}{\mathbb R}
\newcommand{\N}{\mathbb N}
\newcommand{\EE}{\mathbb E}
\newcommand{\PP}{\mathbb P}
\newcommand{\dd}{\,\mathrm d}
\newcommand{\one}{\mathbf 1}
\newcommand{\norm}[1]{\lVert#1\rVert}

\newcommand{\ip}[2]{\langle#1,#2\rangle}
\DeclareMathOperator{\Law}{Law}
\DeclareMathOperator{\Var}{Var}

\DeclareMathOperator{\Tr}{Tr}
\DeclareMathOperator{\Gr}{Gr}
\DeclareMathOperator{\spanop}{span}
\DeclareMathOperator{\supp}{supp}

\newcommand{\HS}{\mathrm{HS}}
\newcommand{\TV}{\mathrm{TV}}
\newcommand{\Dsob}{\mathbb D}
\newcommand{\clip}{\operatorname{clip}}
\newcommand{\Risk}{\operatorname{err}}

\setlist[enumerate]{label=\textnormal{(\roman*)},leftmargin=2.2em,itemsep=2pt,topsep=4pt}
\allowdisplaybreaks[1]
\hypersetup{
 pdftitle={Optimal information loss for Liouville spectral responses},
 pdfauthor={Chunhao Cai},
 pdfsubject={Liouville spectral responses, Fisher information, and optimal observations}
}

\title{Optimal information loss for Liouville spectral responses}
\author{Chunhao Cai\\[2pt]
{\small School of Mathematics (Zhuhai), Sun Yat-Sen University}\\
{\small \texttt{caichh9@mail.sysu.edu.cn}}}
\date{}

\begin{document}
\maketitle

\begin{abstract}
We study Fisher information about the amplitude of a known smooth mean shift
from the first Dirichlet eigenvalue and spatial responses of a Liouville random
medium.  The complete response record determines the speed measure and the
Gaussian free field, whereas every finite record strictly loses information;
an optimal design exists in each finite-dimensional probe space.  We quantify
this gap as the known Liouville parameter $\gamma$ tends to zero.  If the probe
space contains an explicit matching probe, the minimum loss equals $\gamma^2$
times a positive compression error for a deterministic second-chaos operator,
up to $o(\gamma^2)$, with designs allowed to depend on $\gamma$.  After
enlarging the probe dictionary, the optimal $m$-response coefficient is
asymptotic to $\|Lg\|_{L^1(D)}^2/(8m)$, where
$L=-(2\pi)^{-1}\Delta_D$ on the unit disk $D$.  A quadratic correction of the
readouts attains the same fixed-design loss to second order.
\end{abstract}



\tableofcontents
\clearpage

\section{Introduction}
\label{sec:introduction}

Spectral responses record how eigenvalues change under perturbations of a
speed measure.  We study them for Liouville Brownian motion, whose construction
and Dirichlet form are developed in
\cite{GRV2016,GRV2014,Berestycki2015} and whose spectral geometry is studied in
\cite{BerestyckiWong,BerestyckiKlein,BerestyckiSpectral}.  The response calculus
of \cite{CaiResponse} treats smooth multiplicative perturbations of the speed
measure; these responses are the observations used below.

Statistically, the problem is constrained score compression.  A known mean
direction has unknown amplitude, while only finitely many nonlinear spectral
readouts may be retained.  The conditional-score identity
\cite[Theorem~7, arXiv v2]{Pollard} expresses the lost Fisher information as a
conditional variance.  This connects the problem with finite-summary
compression and conditional-variance reduction
\cite{AlsingWandelt,Efron2018,LiuIhler,ConstantineEtAl,LeeActive,KunduWycoff};
related questions for random fields and nonlinear observations appear in
\cite{Carron,PietroniSchmidt,NikolacSchmidtTucci,ChenHosseiniOwhadiStuart}.
Here, however, the admissible summaries are fixed by the spectrum of a singular
random medium.

The complete response record is the lossless benchmark.  Its responses first
determine the ground-state-weighted speed measure; a reciprocal monotonicity
argument removes the ground-state weight, and the reconstruction theorem of
Berestycki, Sheffield and Sun \cite{BerestyckiSheffieldSun} then recovers the
Gaussian free field.  This mechanism is related to inverse recovery from
principal-eigenvalue data \cite{SantacesariaYachimura,ChoulliTriki} and to the
reciprocal monotonicity argument in \cite{OlivaPetitta}.

The three main results form one information hierarchy.  First, the complete
spatial response record is lossless: it recovers the speed measure and the
field.  Second, every fixed finite response record lies strictly below this
benchmark, although an increasing dense dictionary exhausts the lost
information.  This leaves the quantitative question of how small the loss can
be at a fixed response budget.  In the weak-disorder limit, an explicit probe
matches the full Gaussian score at first order, and a deterministic
second-order operator governs the remaining loss.  Gaussian-chaos conditional
projection \cite{Janson,AlpayKipnis} and negative-order operator asymptotics
\cite{Ponge,RozenblumShargorodsky} then yield the sharp large-budget
coefficient.

\subsection{The medium, the parameter, and the observations}
\label{subsec:model}

Let
\[
 D=\{x\in\R^2:|x|<1\},\qquad
 L=-\frac1{2\pi}\Delta_D,\qquad G=L^{-1},
\]
where \(\Delta_D\) is the Dirichlet Laplacian, and set
\begin{equation*}
 E(u,v)=\frac1{2\pi}\int_D\nabla u\cdot\nabla v\dd x.
\end{equation*}
Write
\[
 H=L^2(D),\qquad H_E=H_0^1(D),\qquad \|u\|_E^2=E(u,u),
\]
and write \(\|\cdot\|\) for the norm of \(H\).  We realize the zero-boundary Gaussian free field \(h\) in
\(H^{-1}(D)=H_E^*\), with covariance operator \(G\).  Let \(\mathcal M(D)\)
denote the finite positive Radon measures on \(D\), endowed with the narrow
topology.

With \(h_\varepsilon=e^{-\varepsilon L}h\), define the subcritical Liouville
measure by
\begin{equation}
 \mu_\gamma(\dd x)
 =\lim_{\varepsilon\downarrow0}
 \exp\!\left(
   \gamma h_\varepsilon(x)
   -\frac{\gamma^2}{2}\EE h_\varepsilon(x)^2
 \right)\dd x,
 \qquad 0<\gamma<2,
 \label{eq:gmc}
\end{equation}
where the limit is in probability for narrow convergence of finite measures.
No normalization by the total mass is applied, and all values of \(\gamma\) are coupled through the same field.

Let \((\lambda_\gamma,\phi_\gamma)\) be the positive normalized ground-state
pair of the Liouville speed-measure operator:
\[
 E(\phi_\gamma,v)
 =\lambda_\gamma\int_D\phi_\gamma v\dd\mu_\gamma,
 \qquad
 \int_D\phi_\gamma^2\dd\mu_\gamma=1,
 \qquad \phi_\gamma>0\quad\mu_\gamma\text{-a.e.}
\]
for \(v\in H_0^1(D)\cap L^2(\mu_\gamma)\).  At \(\gamma=0\), write
\[
 L\psi=\lambda_0\psi,
 \qquad \psi>0,
 \qquad \|\psi\|=1.
\]

Constants produce deterministic logarithmic responses, so we work with the
centered probe space
\[
 \mathcal P
 =\left\{f\in C^\infty(\overline D):
          \int_D f\psi^2\dd x=0\right\}.
\]
For \(f\in\mathcal P\), let
\[
 A_f(\mu)
 =-\left.\frac{\mathrm d}{\mathrm dt}
          \log\lambda_1(e^{tf}\mu)\right|_{t=0},
 \qquad
 A_{f,\gamma}=A_f(\mu_\gamma).
\]
The deterministic simple-eigenvalue response formula \cite[Corollary~3.6]{CaiResponse}, applied to the speed-measure tilt $e^{tf}\mu$, gives
\[
 A_{f,\gamma}=\int_D f\dd\nu_\gamma,
 \qquad
 \nu_\gamma=\phi_\gamma^2\mu_\gamma.
\]
Thus \(\nu_\gamma\) is a probability measure.

Fix a known \(0\ne g\in C_c^\infty(D)\).  On \(H^{-1}(D)\), consider the
Cameron--Martin shift experiment
\[
 P_\theta=\Law(h+\theta g),\qquad \theta\in\R.
\]
Under \(P_0\), let \(W\) be isonormal on \(H\) and write
\(h=G^{1/2}W\) distributionally.  Set
\[
 z=L^{1/2}g,\qquad
 Z=W(z)=\langle h,Lg\rangle,
 \qquad
 a=\|z\|^2=E(g,g)>0.
\]
Then
\begin{equation}
 \frac{\dd P_\theta}{\dd P_0}
 =\exp\!\left(\theta Z-\frac12a\theta^2\right).
 \label{eq:complete-likelihood}
\end{equation}
Throughout, \(\gamma\) is known and Fisher information for \(\theta\) is
evaluated at \(\theta=0\).  We write \(\PP=P_0\).  Unless stated
otherwise, all expectations and conditional expectations are under \(P_0\).

Choose nested spaces \(V_K\subset\mathcal P\), \(\dim V_K=K\), such that
\[
 \overline{\bigcup_{K\ge1}V_K}^{\,\|\cdot\|_\infty}
 =\left\{f\in C(\overline D):\int_Df\psi^2\dd x=0\right\}.
\]
For any finite-dimensional \(V\subset V_K\) and any basis
\(f_1,\ldots,f_{\dim V}\) of \(V\), define
\[
 S_{\gamma,V}
 =\bigl(\lambda_\gamma,A_{f_1,\gamma},\ldots,
          A_{f_{\dim V},\gamma}\bigr).
\]
We use \(S_{\gamma,V}\) for any such coordinate representation; all
quantities below depend only on the sigma-algebra it generates.  Designs are
deterministic, although they may depend on the known \(g\), \(\gamma\), and
\(V_K\).  The conditional-score identity gives
\begin{align*}
 I_\gamma(V;g)
 &=\EE\!\left[\EE(Z\mid S_{\gamma,V})^2\right],\\
 \ell_\gamma(V;g)
 &:=a-I_\gamma(V;g)
 =\EE\!\left[(Z-\EE(Z\mid S_{\gamma,V}))^2\right]\\
 &=\inf_{b\in L^2(\Law(S_{\gamma,V}))}
      \EE\!\left[(Z-b(S_{\gamma,V}))^2\right].
\end{align*}
For \(1\le m\le K\), define
\[
 \Delta_{\gamma,m,K}(g)
 =\inf_{V\in\Gr(m,V_K)}\ell_\gamma(V;g).
\]

\subsection{Complete observations}
\label{subsec:intro-complete}

Choose \((f_j)_{j\ge1}\subset\mathcal P\) so that
\(V_K=\spanop(f_1,\ldots,f_K)\), and set
\[
 S_{\gamma,\infty}
 =\bigl(\lambda_\gamma,(A_{f_j,\gamma})_{j\ge1}\bigr).
\]
Let \(\mathcal S=(0,\infty)\times\R^\N\).  A reconstruction pair is a pair
of measurable maps
\[
 R_\mu:\mathcal S\to\mathcal M(D),
 \qquad
 R_h:\mathcal S\to H^{-1}(D).
\]
It may depend on the known \(\gamma\), but not on \(\theta\).

\begin{theorem}
\label{thm:complete}
For every \(0<\gamma<2\), there exists a reconstruction pair
\((R_\mu,R_h)\) such that, for every \(\theta\in\R\),
\begin{equation}
 R_\mu(S_{\gamma,\infty})=\mu_\gamma,
 \qquad
 R_h(S_{\gamma,\infty})=h,
 \qquad P_\theta\text{-a.s.}
 \label{eq:main-reconstruction-maps}
\end{equation}
Moreover, as \(K\to\infty\),
\begin{equation}
 \EE[Z\mid S_{\gamma,V_K}]\longrightarrow Z
 \quad\text{in }L^2(P_0),
 \qquad
 \ell_\gamma(V_K;g)\downarrow0.
 \label{eq:main-complete}
\end{equation}
\end{theorem}

Thus the complete response record is equivalent to observing the field for the
mean-shift experiment.

\subsection{Finite resources}
\label{subsec:intro-finite}

For the strict upper bound on the loss, set
\begin{equation}
 \bar g=g-\int_Dg\psi^2\dd x\in\mathcal P.
 \label{eq:centered-target-probe}
\end{equation}

\begin{theorem}
\label{thm:finite}
For every \(0<\gamma<2\) and \(1\le m\le K<\infty\),
\[
 V\longmapsto I_\gamma(V;g)
 \quad\text{is continuous on }\Gr(m,V_K),
\]
and
\begin{equation}
 0<\ell_\gamma(V_K;g)
 \le \Delta_{\gamma,m,K}(g)
 =\min_{V\in\Gr(m,V_K)}\ell_\gamma(V;g).
 \label{eq:main-finite}
\end{equation}
If \(\bar g\in V_K\), then \(\Delta_{\gamma,m,K}(g)<a\).
\end{theorem}

Combining Theorems~\ref{thm:complete} and~\ref{thm:finite}, for fixed
\(0<\gamma<2\),
\begin{equation*}
 0<\Delta_{\gamma,K,K}(g)
 =\ell_\gamma(V_K;g)\downarrow0
 \qquad(K\to\infty).
\end{equation*}

\subsection{Optimal weak-disorder loss}
\label{subsec:intro-quantitative}

We now keep the number of response readouts fixed and let \(\gamma\downarrow0\).
Let
\[
 Q=I-\psi\otimes\psi,
 \qquad
 R_0=(L-\lambda_0)^{-1}Q,
\]
where \(R_0\psi=0\).  For \(f\in\mathcal P\), define
\[
 \rho_f=\psi^2f+2\lambda_0\psi R_0(f\psi),
 \qquad
 v_f=G^{1/2}\rho_f,
 \qquad
 v_0=G^{1/2}\psi^2,
\]
and
\[
 U(V)=\spanop\bigl(v_0,\{v_f:f\in V\}\bigr)\subset H.
\]
After a deterministic invertible normalization, Lemma~\ref{lem:observation-expansion}
identifies \(U(V)\) as the first-order Gaussian observation space.  The score
direction \(z=L^{1/2}g\) is matched by
\begin{equation}
 f_g
 =\frac{(L-\lambda_0)(L+\lambda_0)^{-1}((Lg)/\psi)}{\psi},
 \qquad
 v_{f_g}=z.
 \label{eq:intro-matching}
\end{equation}
Lemma~\ref{lem:matching} proves both \(f_g\in\mathcal P\) and
\eqref{eq:intro-matching}.  In general \(f_g\ne\bar g\): the former matches
the full first-order score, whereas \(\bar g\) is used only to prove positive
information retention for a finite record.

For \(f\in\mathcal P\), let \(C_f\) denote the density Hessian at Lebesgue
measure,
\[
 \ip{u}{C_fv}
 =\left.\partial_s\partial_t
 A_f((1+su+tv)\dd x)\right|_{s=t=0},
 \qquad u,v\in C^\infty(\overline D).
\]
Lemma~\ref{lem:density-hessian} gives its self-adjoint Hilbert--Schmidt
extension to \(H\).  With \(M_u\) denoting multiplication by \(u\), set
\[
 K_g=\frac12G^{1/2}(M_{Lg}+C_{f_g})G^{1/2}.
\]
For a finite-dimensional \(U\subset H\), let \(P_U\) be the orthogonal
projection onto \(U\) and define
\[
 \kappa_g(U)
 =2\left(\|K_g\|_{\HS}^2-\|P_UK_gP_U\|_{\HS}^2\right).
\]
Assume \(f_g\in V_{K_0}\) for some \(K_0\).  For
\(K\ge\max\{m,K_0\}\), put
\[
 \kappa_{m,K}(g)
 =\inf_{\substack{V\in\Gr(m,V_K)\\ f_g\in V}}
     \kappa_g(U(V)),
 \qquad
 \kappa_m(g)=\lim_{K\to\infty}\kappa_{m,K}(g).
\]
The limit is decreasing in \(K\) because the dictionaries are nested.

\begin{theorem}
\label{thm:optimal}
For fixed \(m\ge1\) and \(K\ge\max\{m,K_0\}\),
\begin{equation}
 \Delta_{\gamma,m,K}(g)
 =\gamma^2\kappa_{m,K}(g)+o(\gamma^2),
 \qquad
 0<\kappa_{m,K}(g)<\infty,
 \qquad \gamma\downarrow0.
 \label{eq:main-optimal}
\end{equation}
Moreover,
\begin{equation}
 \lim_{m\to\infty}m\kappa_m(g)
 =\frac18\|Lg\|_{L^1(D)}^2,
 \label{eq:main-budget}
\end{equation}
and hence
\begin{equation}
 \lim_{m\to\infty}m
 \lim_{K\to\infty}
 \lim_{\gamma\downarrow0}
 \frac{\Delta_{\gamma,m,K}(g)}{\gamma^2}
 =\frac18\|Lg\|_{L^1(D)}^2.
 \label{eq:main-ordered-limit}
\end{equation}
\end{theorem}

The optimization in \eqref{eq:main-optimal} precedes the weak-disorder
limit, so the design may depend on \(\gamma\); asymptotically optimal designs
approach \(f_g\in V\), leaving the compression error of \(K_g\).

For a fixed matched design \(V\), Section~\ref{subsec:score-correction}
constructs a quadratic function \(\widehat Z_{\gamma,V}\) of the normalized
readouts such that
\[
 \|\widehat Z_{\gamma,V}-\EE[Z\mid S_{\gamma,V}]\|_{L^2(P_0)}
 =o(\gamma).
\]
Thus the same second-order coefficient describes both the information loss and
the mean-square prediction error of the corrected spectral score.

The limits in \eqref{eq:main-ordered-limit} are taken in the displayed
order; no joint limit is asserted.

\section{Preliminaries}
\label{sec:preliminaries}

We use real Hilbert spaces.  For a compact operator $B$, let $s_j(B)$ denote
its singular values in nonincreasing order, counted with multiplicity, and let
$\mathfrak S_p$ be the class defined by $\sum_js_j(B)^p<\infty$.  For $u,v\in H$,
$u\otimes v$ denotes the rank-one operator
$(u\otimes v)x=u\ip{v}{x}$; the same convention is used on $H_E$.
On an $L^2$ space, $M_f$ denotes multiplication by $f$.  For a random
variable $F$ and finite $p\ge1$, write
$\|F\|_p=(\EE|F|^p)^{1/p}$, using the norm of its target space when needed.
Unqualified probability norms and convergence statements refer to $P_0$.
For a random variable $Y$ with values in a standard Borel space,
$\sigma(Y)$ denotes its generated sigma-algebra completed under $P_0$.
The spaces $\mathcal M(D)$ and $H^{-1}(D)$ carry the Borel sigma-algebras
specified in Section~\ref{subsec:model}.  We write $\mathrm{Leb}_d$ for
Lebesgue measure on $\R^d$.

\subsection{Liouville speed measures and Green operators}
\label{subsec:green-input}

A finite positive measure $\eta$ has finite energy if
\[
 \int_D|u|\dd\eta\le C\|u\|_E,
 \qquad u\in C_c^\infty(D),
\]
for some $C<\infty$; a signed measure has finite energy when its total
variation does.

\begin{lemma}
\label{lem:finite-energy-potential-input}
For every finite-energy signed measure $\eta$, there is a unique
$G\eta\in H_E$ such that
\begin{equation*}
 E(G\eta,u)=\int_D\widetilde u\dd\eta,
 \qquad u\in H_E.
\end{equation*}
\end{lemma}

\begin{proof}
The proof is given in \cite[Lemma~6.14]{CaiResponse}.
\end{proof}

We use the Liouville Dirichlet-form construction of
Garban, Rhodes and Vargas \cite{GRV2016,GRV2014}.  The following statements
are taken from \cite{CaiResponse} and are written in the notation of the
present paper.

\begin{proposition}
\label{prop:liouville-measure-input}
For every $0<\gamma<2$, almost surely,
\[
 \mu_\gamma(D)<\infty,
 \qquad
 \mu_\gamma\text{ is smooth and has full }E\text{-quasi-support}.
\]
Moreover,
\[
 \mu_\gamma(O)>0\quad\text{for every nonempty open }O\subset D,
 \qquad
 \mu_\gamma\perp\dd x.
\]
\end{proposition}
\begin{proof}
The smoothness and full quasi-support are given by
\cite[Proposition~6.7]{CaiResponse}, full support by
\cite[Lemma~6.16]{CaiResponse}, and mutual singularity with Lebesgue
measure by \cite[Lemma~6.28]{CaiResponse}.
\end{proof}

For a field configuration $\omega\in H^{-1}(D)$, write
$\mu_\gamma^\omega$ for the measurable GMC version constructed in
\cite[Lemma~6.2]{CaiResponse}, and let
$\Omega_{\rm GMC}\subset H^{-1}(D)$ be the corresponding full-probability
set.  Then $\mu_\gamma^h=\mu_\gamma$ almost surely.

\begin{theorem}
\label{thm:gmc-shift-input}
For every $v\in C_c^\infty(D)$,
\[
 \mu_\gamma^{\omega+v}=e^{\gamma v}\mu_\gamma^\omega,
 \qquad \omega,\omega+v\in\Omega_{\rm GMC}.
\]
In particular, for each fixed $v\in C_c^\infty(D)$,
$\mu_\gamma^{h+v}=e^{\gamma v}\mu_\gamma^h$ almost surely.
\end{theorem}
\begin{proof}
This is \cite[Theorem~6.3]{CaiResponse}, specialized to
$C_c^\infty(D)$ Cameron--Martin directions.
\end{proof}

For $\mu=\mu_\gamma$, define
\begin{equation}
 (G_\mu u)(x)=\int_DG(x,y)u(y)\dd\mu(y),
 \qquad u\in L^2(\mu).
 \label{eq:green-operator}
\end{equation}

\begin{proposition}
\label{prop:green-operator-input}
Almost surely, $G_\mu$ is a positive injective self-adjoint
Hilbert--Schmidt operator on $L^2(\mu)$.  For every $u\in L^2(\mu)$,
$u\mu$ has finite energy, and $G_\mu$ is the inverse of the Dirichlet
speed-measure generator.
\end{proposition}
\begin{proof}
The Hilbert--Schmidt and finite-energy assertions are contained in
\cite[Lemma~6.17]{CaiResponse}; the inverse relation is
\cite[Lemma~6.18 and Proposition~6.22]{CaiResponse}.
\end{proof}

\begin{proposition}
\label{prop:liouville-spectrum-input}
Almost surely, the Dirichlet speed-measure generator has compact resolvent,
and its positive eigenvalues, listed with multiplicity, satisfy
\[
 0<\lambda_1(\mu)<\lambda_2(\mu)<\cdots\uparrow\infty.
\]
\end{proposition}
\begin{proof}
Compact resolvent and the ordered spectrum follow from
\cite[Corollary~6.25]{CaiResponse}, while strictness is
\cite[Theorem~1.8]{CaiResponse}.
\end{proof}
We also use the measurable ordered eigenvalues and finite-window spectral
projections constructed in
\cite[Proposition~6.22 and Corollary~6.23]{CaiResponse}.

\begin{proposition}
\label{prop:green-smoothing-input}
Almost surely,
\[
 G_\mu:L^2(\mu)\longrightarrow C_b(D),
 \qquad
 \|G_\mu u\|_\infty\le C(\mu)\|u\|_{L^2(\mu)}.
\]
For each $u\in L^2(\mu)$, the integral in \eqref{eq:green-operator} is the
unique continuous representative of $G_\mu u$.
\end{proposition}
\begin{proof}
This is \cite[Proposition~6.26]{CaiResponse}, written in the notation of
\eqref{eq:green-operator}.
\end{proof}

For field reconstruction we use the following theorem of
Berestycki--Sheffield--Sun.

\begin{theorem}
\label{thm:bss-reconstruction-input}
For every $0<\gamma<2$, the Gaussian free field is almost surely measurably
determined by its Liouville measure in the circle-average normalization.
\end{theorem}
\begin{proof}
The reconstruction theorem is \cite[Theorem~1.1, arXiv version]
{BerestyckiSheffieldSun}.
\end{proof}
The normalization change needed for \eqref{eq:gmc} is made at the point of
use in Section~\ref{sec:reconstruction}.

\subsection{Gaussian information and Gaussian calculus}
\label{subsec:gaussian-input}

\begin{proposition}
\label{prop:pollard-score}
Let $(P_\theta)$ be differentiable in quadratic mean with score $s_\theta$,
let $Y$ be measurable, and put $P_\theta^Y=Y_\#P_\theta$.
Then $(P_\theta^Y)$ is differentiable in quadratic mean with score
\[
 \widetilde s_\theta(y)
 =\EE_\theta[s_\theta\mid Y=y].
\]
At each fixed $\theta$, Fisher information is preserved if and only if
\[
 s_\theta=\widetilde s_\theta(Y)
 \qquad P_\theta\text{-a.s.}
\]
\end{proposition}
\begin{proof}
This is the measurable-statistic form of quadratic-mean differentiability
proved in \cite[Theorem~7, arXiv v2]{Pollard}.
\end{proof}

\begin{corollary}
For any measurable observation $Y$ of the experiment
\eqref{eq:complete-likelihood}, the score at $\theta=0$ is
$\EE[Z\mid Y]$, and its Fisher information $I_Y$ satisfies
\[
 I_Y=\EE\!\left[\EE(Z\mid Y)^2\right],
 \qquad
 a-I_Y=\EE\Var(Z\mid Y).
\]
\end{corollary}

\begin{proof}
Apply Proposition~\ref{prop:pollard-score} to the score $Z$ and use
$\EE Z^2=a$ together with conditional variance decomposition.
\end{proof}

For a self-adjoint Hilbert--Schmidt operator $B$ on $H$, $I_2(B)$ is the
centered second Wiener integral, normalized by
\begin{equation}
 \EE I_2(B)I_2(C)=2\Tr(BC).
 \label{eq:second-chaos-normalization}
\end{equation}
For a closed subspace $U\subset H$, Gaussian conditioning preserves chaos
order and
\begin{equation}
 \EE[I_2(B)\mid W(U)]=I_2(P_UBP_U).
 \label{eq:gaussian-conditioning}
\end{equation}
These are standard second-quantization identities; see
\cite[Theorem~4.9]{Janson} and \cite[Section~2]{AlpayKipnis}.

We write $\Dsob^{r,p}$ for Malliavin--Sobolev spaces over $W$, with norm
$\|\cdot\|_{r,p}$, and $\mathrm D$ for the Malliavin derivative.  The same
notation is used for variables with values in a separable Hilbert space.
For $F=(F_1,\ldots,F_d)$, set
\[
 \sigma_F=(\ip{\mathrm D F_i}{\mathrm D F_j})_{i,j=1}^d.
\]
For a Hilbert-valued variable $F_k$ in the $k$th Wiener chaos, $p\ge2$ and
an integer $r\ge0$, hypercontractivity and the standard derivative bounds
\cite{Janson,Nualart} give
\begin{equation}
 \norm{F_k}_{r,p}
 \le C_{r,p}(1+k)^{r/2}(p-1)^{k/2}\norm{F_k}_{L^2}.
 \label{eq:hypercontractivity}
\end{equation}

For total variation we use
$d_{\TV}(\alpha,\beta)=\sup_A|\alpha(A)-\beta(A)|$.  We apply
\cite[Theorem~2.19]{BallyCaramellino} in Lemma~\ref{lem:tv-convergence}.

\section{Reconstruction from the complete response}
\label{sec:reconstruction}

The complete record is the lossless benchmark.  Its reconstruction follows the
chain
\[
 S_{\gamma,\infty}
 \longrightarrow (\lambda_\gamma,\nu_\gamma)
 \longrightarrow \mu_\gamma
 \longrightarrow h.
\]
The first arrow uses the response coordinates, the second removes the unknown
ground-state weight, and the last uses Liouville-measure reconstruction.

For every positive finite-energy measure $\eta$, Lemma~\ref{lem:finite-energy-potential-input}
and $E(|u|,|u|)=E(u,u)$ give
\begin{equation}
 \int_D|\widetilde u|\dd\eta
 =E(G\eta,|u|)
 \le \|G\eta\|_E\|u\|_E,
 \qquad u\in H_E.
 \label{eq:finite-energy-pairing}
\end{equation}

Fix a finite speed measure $\mu$ with full support that charges no set of
zero Dirichlet capacity.  Assume that its Green operator is a positive
injective Hilbert--Schmidt inverse of the speed-measure generator and that
$u\mu$ has finite energy for every $u\in L^2(\mu)$.  Propositions~\ref{prop:liouville-measure-input} and
\ref{prop:green-operator-input} show that these properties hold almost surely
for $\mu=\mu_\gamma$.

Let $(\lambda,\phi)$ be a positive normalized ground-state pair,
\begin{equation}
 E(\phi,v)=\lambda\int_D\phi v\dd\mu,
 \qquad
 \|\phi\|_{L^2(\mu)}=1,
 \qquad
 \phi>0\quad\mu\text{-a.e.},
 \label{eq:inverse-ground-state}
\end{equation}
and set
\begin{equation}
 \nu=\phi^2\mu,
 \qquad
 w=\lambda^{-1/2}\phi.
 \label{eq:inverse-variables}
\end{equation}
Inside integrals against measures charging no set of zero capacity, elements
of $H_E$ are represented by their quasi-continuous versions.  Then
\begin{equation*}
 \nu(D)=1,
 \qquad
 E(w,w)=1,
 \qquad
 E(w,v)=\int_D\frac{v}{w}\dd\nu
 \quad(v\in H_E).
\end{equation*}
Moreover,
\[
 \frac{\nu}{w}=\sqrt\lambda\,\phi\mu
\]
is finite by Cauchy--Schwarz and has finite energy by
Proposition~\ref{prop:green-operator-input}.

\subsection{Eliminating the ground-state weight}
\label{subsec:inverse-uniqueness}

The following monotonicity argument is the reciprocal-measure analogue of
\cite[proof of Theorem~3.1, Appendix~6.1, arXiv v1]{OlivaPetitta}.

\begin{lemma}
\label{lem:inverse-uniqueness}
Let $\nu$ be a probability measure charging no set of zero Dirichlet capacity.
Suppose $w_1,w_2\in H_E$ are positive $\nu$-almost everywhere,
$\nu/w_i$ has finite energy, and
\begin{equation}
 E(w_i,v)=\int_D\frac{v}{w_i}\dd\nu,
 \qquad v\in H_E,
 \qquad i=1,2.
 \label{eq:inverse-two-solutions}
\end{equation}
Then $w_1=w_2$ in $H_E$ and $\nu$-almost everywhere.  Consequently, any
speed measure with ground-state data $(\lambda,\nu)$ as in
\eqref{eq:inverse-ground-state}--\eqref{eq:inverse-variables} is uniquely
recovered by
\begin{equation}
 \mu=\frac1\lambda\frac{\nu}{w^2}.
 \label{eq:inverse-formula}
\end{equation}
\end{lemma}

\begin{proof}
By \eqref{eq:finite-energy-pairing}, the cross integrals below are finite.
Subtracting the two identities in \eqref{eq:inverse-two-solutions} and testing
with $v=w_1-w_2$ gives
\begin{align*}
 E(w_1-w_2,w_1-w_2)
 &=\int_D(w_1-w_2)
   \left(\frac1{w_1}-\frac1{w_2}\right)\dd\nu \\
 &=-\int_D\frac{(w_1-w_2)^2}{w_1w_2}\dd\nu
 \le0.
\end{align*}
Hence $w_1=w_2$ in $H_E$.  Their quasi-continuous representatives agree
outside a set of zero capacity, and therefore $\nu$-almost everywhere.
Finally, $\phi=\sqrt\lambda\,w$ and $\nu=\phi^2\mu$ give
\eqref{eq:inverse-formula}.
\end{proof}

The eigenvalue is needed only to fix scale: replacing $\mu$ by $c\mu$ leaves
$\nu$ unchanged and replaces $\lambda$ by $\lambda/c$.

\subsection{Field reconstruction from the speed measure}
\label{subsec:field-reconstruction}

The normalization in \eqref{eq:gmc} differs from the circle-average
normalization in Theorem~\ref{thm:bss-reconstruction-input} by a deterministic
factor.  We record the resulting measurable reconstruction in the state space
used in this paper.

\begin{lemma}
\label{lem:field-reconstruction}
For every $0<\gamma<2$, there exists a measurable map
\begin{equation}
 \mathcal R_\gamma:\mathcal M(D)\longrightarrow H^{-1}(D)
 \label{eq:field-reconstruction-map}
\end{equation}
such that, for every $\theta\in\R$,
\begin{equation}
 \mathcal R_\gamma(\mu_\gamma)=h
 \qquad P_\theta\text{-a.s.}
 \label{eq:field-reconstruction-identity}
\end{equation}
\end{lemma}

\begin{proof}
Let $r_D$ be the conformal radius.  By approximation independence of
subcritical GMC, used for the heat-semigroup construction in
\cite[Lemma~6.2]{CaiResponse}, the measure in \eqref{eq:gmc} agrees with the
variance-normalized circle-average chaos.  Since
\[
 \Var(h_\varepsilon^{\rm circ}(x))
 =\log \varepsilon^{-1}+\log r_D(x),
\]
the circle-average normalization of \cite{BerestyckiSheffieldSun},
$\varepsilon^{\gamma^2/2}e^{\gamma h_\varepsilon^{\rm circ}(x)}\dd x$,
satisfies
\begin{equation*}
 \mu_\gamma^{\rm circ}(\dd x)
 =r_D(x)^{\gamma^2/2}\mu_\gamma(\dd x).
\end{equation*}
Thus $\mu_\gamma$ measurably determines $\mu_\gamma^{\rm circ}$, and
Theorem~\ref{thm:bss-reconstruction-input} measurably determines every smooth
test pairing of $h$ from $\mu_\gamma$.

Choose a countable dense set $(\zeta_j)_{j\ge1}\subset C_c^\infty(D)$ in $H_E$ and define
\[
 \iota:H^{-1}(D)\longrightarrow\R^\N,
 \qquad
 \iota(u)=(\langle u,\zeta_j\rangle)_{j\ge1}.
\]
The map $\iota$ is continuous and injective.  By the Lusin--Souslin theorem,
$\iota(H^{-1}(D))$ is a Borel subset of $\R^\N$, and $\iota^{-1}$ is
measurable on that image.  Let $(\psi_j)_{j\ge1}$ be an $H$-orthonormal Dirichlet eigenbasis,
$L\psi_j=\lambda_j\psi_j$.  Then
\[
 h=\sum_{j\ge1}\lambda_j^{-1/2}\xi_j\psi_j
 \quad\text{in }H^{-1}(D),
 \qquad
 \EE\|h\|_{H^{-1}(D)}^2
 =\sum_{j\ge1}\lambda_j^{-2}
 =\Tr(G^2)<\infty,
\]
where the last finiteness follows from the two-dimensional Dirichlet Weyl law.
Thus the reconstructed sequence of test pairings belongs to $\iota(H^{-1}(D))$
almost surely.  Applying $\iota^{-1}$ there and a fixed value off the image gives
\eqref{eq:field-reconstruction-map}--\eqref{eq:field-reconstruction-identity}
under $P_0$.  Cameron--Martin equivalence gives the same identity under every
$P_\theta$.
\end{proof}

\subsection{Final proof of Theorem~\ref{thm:complete}}
\label{subsec:final-complete}

\begin{proof}
\emph{Recovery of $\nu_\gamma$.}
The map
\[
 \nu\longmapsto\left(\int_D f_j\dd\nu\right)_{j\ge1}
 \in\R^\N
\]
from probability measures on $D$ is measurable and injective.  To prove
injectivity, let probability measures $\nu_1,\nu_2$ have the same image.  For
$f\in C(\overline D)$, set
$f^\circ=f-\int_Df\psi^2\dd x$.  Choose
$q_n\in\bigcup_KV_K$ with $\|q_n-f^\circ\|_\infty\to0$.  Since
$(\nu_1-\nu_2)(D)=0$ and every $q_n$ is a finite linear combination of the
$f_j$,
\begin{equation*}
 \left|\int_D f\dd(\nu_1-\nu_2)\right|
 \le2\|f^\circ-q_n\|_\infty\longrightarrow0.
\end{equation*}
Hence $\nu_1=\nu_2$.  By the Lusin--Souslin theorem, the image is Borel in
$\R^\N$ and the inverse is measurable there.  Extend the inverse off the image
by the fixed measure $\psi^2\dd x$.  Thus the response sequence determines
$\nu_\gamma$ measurably.

\emph{Recovery of $\mu_\gamma$.}
By the measurable spectral construction in
\cite[Proposition~6.22 and Corollaries~6.23--6.24]{CaiResponse}, together with \cite[Theorem~1.8]{CaiResponse}, all coordinates of
$S_{\gamma,\infty}$ are measurable functions of the field.  Lemma~\ref{lem:field-reconstruction}
therefore makes $S_{\gamma,\infty}$ a measurable function of $\mu_\gamma$.
Let $\pi_s$ be a regular conditional law of $\mu_\gamma$ given
$S_{\gamma,\infty}=s$.  Disintegration of the probability-one structural
event fixed at the start of Section~\ref{sec:reconstruction} shows that, for
almost every $s$, $\pi_s$ is supported on speed measures satisfying those
hypotheses.  Two independent draws $\mu_1,\mu_2\sim\pi_s$ then have the same
ground-state eigenvalue and the same response measure $\nu$ reconstructed
from $s$.  Lemma~\ref{lem:inverse-uniqueness} therefore gives
\begin{equation*}
 (\pi_s\otimes\pi_s)
 \{(\mu_1,\mu_2):\mu_1=\mu_2\}=1.
\end{equation*}
Thus $\pi_s$ is a Dirac measure for almost every $s$.  Since
$\mathcal M(D)$ is a standard Borel space, the Dirac embedding
$\mu\mapsto\delta_\mu$ is a Borel isomorphism onto its image.  Applying this
inverse to $s\mapsto \pi_s$, and fixing an arbitrary finite measure on the
exceptional set, gives a measurable map
\begin{equation*}
 R_\mu:\mathcal S\longrightarrow\mathcal M(D),
 \qquad
 R_\mu(S_{\gamma,\infty})=\mu_\gamma
 \quad P_0\text{-a.s.}
\end{equation*}
Set $R_h=\mathcal R_\gamma\circ R_\mu$.  Lemma~\ref{lem:field-reconstruction} gives
\eqref{eq:main-reconstruction-maps} under $P_0$, and Cameron--Martin
equivalence gives the same identities under every $P_\theta$.

\emph{Exhaustion of information.}
Conversely, the complete record and $\mu_\gamma$ are measurable functions of
$h$.  Therefore
\begin{equation*}
 \sigma(S_{\gamma,\infty})
 =\sigma(\mu_\gamma)
 =\sigma(h)
 \qquad\text{modulo }P_0\text{-null sets}.
\end{equation*}
Because $V_K=\spanop(f_1,\ldots,f_K)$,
\[
 \sigma(S_{\gamma,V_K})\uparrow\sigma(S_{\gamma,\infty}).
\]
Since $Z\in L^2(P_0)$, the $L^2$ martingale convergence theorem yields
\[
 \EE[Z\mid S_{\gamma,V_K}]\longrightarrow Z
 \quad\text{in }L^2(P_0),
\]
and hence
\[
 \ell_\gamma(V_K;g)
 =\|Z-\EE[Z\mid S_{\gamma,V_K}]\|_2^2
 \downarrow0.
\]
This is \eqref{eq:main-complete}.
\end{proof}

\section{Strict loss for finite observations}
\label{sec:finite-loss}

Section~\ref{sec:reconstruction} identifies the complete record with the full
field.  We now show that no finite response record is sufficient.  The key is a
joint density for the score and finitely many readouts; the second response
provides the full-rank derivatives needed for the Gaussian slicing argument.

\subsection{Second response and its source measure}
\label{subsec:response-measures}

Fix a realization $\mu=\mu_\gamma$ for which the conclusions of
Propositions~\ref{prop:liouville-measure-input}--\ref{prop:green-smoothing-input}
hold.  Write
\[
 r_j=\lambda_j(\mu)^{-1},\qquad
 r=r_1>r_2\ge r_3\ge\cdots>0,
 \qquad
 c_j(f)=\int_D f\phi_1\phi_j\dd\mu.
\]
Here $(\phi_j)_{j\ge1}$ is a real orthonormal eigenbasis of $L^2(\mu)$ with
$\phi_1>0$ $\mu$-almost everywhere, and
\[
 \nu=\phi_1^2\mu.
\]
The calculation below is deterministic and applies verbatim to every bounded
smooth tilt of such a realization.  For $f\in\mathcal P$ and
$v\in C_c^\infty(D)$, set
\[
 \partial_vA_{f,\gamma}
 :=\left.\frac{\dd}{\dd t}A_f(e^{\gamma tv}\mu)\right|_{t=0}.
\]

\begin{lemma}
\label{lem:response-measures}
For $f\in\mathcal P$ and $v\in C_c^\infty(D)$,
\begin{equation}
 \partial_vA_{f,\gamma}
 =\gamma\sum_{j\ge2}\frac{r+r_j}{r-r_j}c_j(f)c_j(v)
 =\int_Dv\dd\beta_f,
 \label{eq:second-response}
\end{equation}
where
\begin{equation}
 \beta_f
 =\gamma\phi_1
   \left(\sum_{j\ge2}\frac{r+r_j}{r-r_j}c_j(f)\phi_j\right)\mu.
 \label{eq:beta-definition}
\end{equation}
The series in \eqref{eq:beta-definition} defines a finite signed measure,
$\beta_f(D)=0$, and
\begin{equation}
 \int_Df\dd\beta_f
 \ge\gamma\Var_{\nu}(f)>0,
 \qquad f\in\mathcal P\setminus\{0\}.
 \label{eq:strict-response}
\end{equation}
\end{lemma}

\begin{proof}
For $q\in C^\infty(\overline D)$, put $\mu_t=e^{tq}\mu$ and transport the
Green operator from $L^2(\mu_t)$ to $L^2(\mu)$:
\[
 B_q(t)=M_{e^{tq/2}}G_\mu M_{e^{tq/2}}.
\]
Then
\[
 B_q'(0)=\tfrac12(M_qG_\mu+G_\mu M_q),\qquad
 B_q''(0)=\tfrac14(M_q^2G_\mu+2M_qG_\mu M_q+G_\mu M_q^2).
\]
Let $r(t)$ be the largest eigenvalue of $B_q(t)$.  Since $r=r(0)$ is
simple, the standard perturbation formula gives
\begin{align*}
 r'(0)&=r c_1(q),\\
 \langle\phi_j,B_q'(0)\phi_1\rangle_{L^2(\mu)}
 &=\tfrac12(r+r_j)c_j(q),\\
 r''(0)
 &=\langle\phi_1,B_q''(0)\phi_1\rangle_{L^2(\mu)}
   +2\sum_{j\ge2}
    \frac{|\langle\phi_j,B_q'(0)\phi_1\rangle_{L^2(\mu)}|^2}{r-r_j}\\
 &=r c_1(q)^2
   +r\sum_{j\ge2}\frac{r+r_j}{r-r_j}c_j(q)^2.
\end{align*}
Hence
\[
 \left.\frac{\dd^2}{\dd t^2}\log r(t)\right|_{t=0}
 =\sum_{j\ge2}\frac{r+r_j}{r-r_j}c_j(q)^2.
\]
Polarization yields
\[
 -\left.\partial_s\partial_t
 \log\lambda_1(e^{sf+tv}\mu)\right|_{s=t=0}
 =\sum_{j\ge2}\frac{r+r_j}{r-r_j}c_j(f)c_j(v).
\]
Since
$A_f(\eta)=-\left.\partial_s\log\lambda_1(e^{sf}\eta)\right|_{s=0}$,
replacing $t$ by $\gamma t$ gives \eqref{eq:second-response}.

Moreover,
\[
 1\le\frac{r+r_j}{r-r_j}\le\frac{r+r_2}{r-r_2},\qquad
 \sum_{j\ge1}c_j(f)^2=\|f\phi_1\|_{L^2(\mu)}^2<\infty.
\]
Thus the series in \eqref{eq:beta-definition} converges in $L^2(\mu)$.
Cauchy--Schwarz gives $|\beta_f|(D)<\infty$, and orthogonality to $\phi_1$
gives $\beta_f(D)=0$.  Finally,
\begin{align*}
 \int_Df\dd\beta_f
 &=\gamma\sum_{j\ge2}\frac{r+r_j}{r-r_j}c_j(f)^2\\
 &\ge\gamma\sum_{j\ge2}c_j(f)^2
 =\gamma\left(
       \int_Df^2\dd\nu
        -\left(\int_Df\dd\nu\right)^2\right).
\end{align*}
The measure $\nu$ has full support.  A nonzero
$f\in\mathcal P$ is nonconstant, so the variance is strictly positive.
\end{proof}

\subsection{A mixed density for scores and responses}
\label{subsec:mixed-density}

The source measures in Lemma~\ref{lem:response-measures} provide the
full-rank condition for finite-dimensional Gaussian slicing.

\begin{lemma}
\label{lem:coherent-finite-slice}
Let $w_1,\ldots,w_d\in C_c^\infty(D)$ be energy-orthonormal and write
\[
 h=h^\perp+\sum_{j=1}^d\tau_jw_j,
 \qquad \tau\sim N(0,I_d),
\]
with $\tau$ independent of $h^\perp$.  For almost every $h^\perp$ there is a
finite measure $\mu_\gamma^\perp$ such that
\begin{equation}
 \mu_t
 :=e^{\gamma\sum_{j=1}^dt_jw_j}\mu_\gamma^\perp,
 \qquad t\in\R^d,
 \label{eq:coherent-slice}
\end{equation}
defines a coherent family satisfying
\[
 \mu_t
 =\mu_\gamma^{\,h^\perp+\sum_{j=1}^dt_jw_j}
 \quad\text{for }N(0,I_d)\text{-a.e. }t.
\]
After transport to a fixed $L^2$ space, $G_{\mu_t}$,
$\lambda_1(\mu_t)$, and $A_f(\mu_t)$ for $f\in\mathcal P$ are real analytic
in $t$.
\end{lemma}

\begin{proof}
Choose a measurable full-probability set
$\Omega_0\subset\Omega_{\rm GMC}$ on which the conclusions of
Propositions~\ref{prop:liouville-measure-input}--\ref{prop:green-smoothing-input}
hold.  Disintegration with respect to $(h^\perp,\tau)$ gives, for almost every
$h^\perp$, a full $N(0,I_d)$-measure set of $t\in\R^d$ for which
\[
 h^\perp+\sum_{j=1}^dt_jw_j\in\Omega_0.
\]
Fix $t_0$ in this set and set
\[
 \mu_\gamma^\perp
 =e^{-\gamma\sum_jt_{0,j}w_j}
   \mu_\gamma^{\,h^\perp+\sum_jt_{0,j}w_j}.
\]
For every $t$ in the same full-measure set, Theorem~\ref{thm:gmc-shift-input}, applied with
$v=\sum_j(t_j-t_{0,j})w_j$, gives the asserted agreement between
\eqref{eq:coherent-slice} and the canonical GMC.

Let
\[
 \eta=\mu_\gamma^{\,h^\perp+\sum_jt_{0,j}w_j}.
\]
For bounded smooth $q$, multiplication by $e^{q/2}$ defines a unitary map
$U_q:L^2(e^q\eta)\to L^2(\eta)$ and
\[
 U_qG_{e^q\eta}U_q^{-1}
 =M_{e^{q/2}}G_\eta M_{e^{q/2}}.
\]
With $q=\gamma\sum_j(t_j-t_{0,j})w_j$, the right-hand side is analytic in
operator norm.  Since $G(x,y)>0$ on $D\times D$ and $\eta$ has full support,
every bounded positive tilt is positivity improving, so the largest Green
eigenvalue is simple.  Analytic perturbation gives $\lambda_1(\mu_t)$
analytically in $t$.  The two-parameter family associated with
\[
 e^{\gamma\sum_j(t_j-t_{0,j})w_j+sf}\eta
\]
is analytic in $(t,s)$ and has a simple top Green eigenvalue.  Hence
\[
 A_f(\mu_t)
 =-\left.\partial_s\log\lambda_1(e^{sf}\mu_t)\right|_{s=0}
\]
is analytic in $t$.
\end{proof}

\begin{lemma}
\label{lem:mixed-density}
Let $f_1,\ldots,f_m\in\mathcal P$ be linearly independent and let
$g_1,\ldots,g_p\in C_c^\infty(D)$ be linearly independent.  Put
$Z_i=W(L^{1/2}g_i)$.  Then
\begin{equation}
 \Law\bigl(Z_1,\ldots,Z_p,
      \lambda_\gamma,A_{f_1,\gamma},\ldots,A_{f_m,\gamma}\bigr)
 \ll\mathrm{Leb}_{p+m+1}.
 \label{eq:mixed-density}
\end{equation}
\end{lemma}

\begin{proof}
It suffices to prove absolute continuity after replacing $\lambda_\gamma$
by $\log\lambda_\gamma$.  For $v\in C_c^\infty(D)$, the directional
derivatives of
\[
 Y=(Z_1,\ldots,Z_p,\log\lambda_\gamma,
        A_{f_1,\gamma},\ldots,A_{f_m,\gamma})
\]
are obtained by pairing $v$ with the finite signed measures
\begin{equation}
 (Lg_1)\dd x,\ldots,(Lg_p)\dd x,
 -\gamma\nu_\gamma,\beta_{f_1},\ldots,\beta_{f_m}.
 \label{eq:mixed-sources}
\end{equation}
The first $p$ measures are absolutely continuous with respect to Lebesgue
measure, whereas the last $m+1$ are absolutely continuous with respect to
$\mu_\gamma\perp\dd x$.  Hence a linear relation among
\eqref{eq:mixed-sources} splits into these two classes.

If $u=\sum_i c_ig_i$ and $\sum_i c_iLg_i=0$, then
\[
 E(u,u)=\int_Du\,Lu\dd x=0,
\]
so $u=0$ and $c_1=\cdots=c_p=0$.  If
$b_0\nu_\gamma+\sum_jb_j\beta_{f_j}=0$, total mass gives $b_0=0$.
With $f=\sum_jb_jf_j\in\mathcal P$, linearity gives $\beta_f=0$, and
\eqref{eq:strict-response} gives $f=0$.  Thus $b_1=\cdots=b_m=0$, so the
measures in \eqref{eq:mixed-sources} are linearly independent almost surely.

Choose a countable $\mathbb Q$-vector space
$\mathcal V\subset C_c^\infty(D)$ that is uniformly dense in $C_0(D)$.
For $v\in\mathcal V$, define
\[
 D_vY(h)=\lim_{n\to\infty}n\{Y(h+n^{-1}v)-Y(h)\},
\]
with value $0$ when the finite limit does not exist.  The set on which the
limit exists is measurable, so $D_vY$ is measurable.  By the measurable spectral construction cited after
Proposition~\ref{prop:liouville-spectrum-input}, the spectral coordinates in
$Y$ agree almost surely with their pathwise definitions.  Cameron--Martin
equivalence allows us to intersect this event, $\Omega_{\rm GMC}$, and all
of their translates by $n^{-1}v$, $v\in\mathcal V$, $n\ge1$.  On the
resulting probability-one event, Theorem~\ref{thm:gmc-shift-input}
identifies the difference quotient defining $D_vY$ with the corresponding
difference quotient along the analytic speed-measure tilt.  Hence $D_vY$ equals the analytic
directional derivative for every $v\in\mathcal V$ almost surely.

Let $d=p+m+1$.  If no $\mathbf v=(v_1,\ldots,v_d)\in\mathcal V^d$ gave an
invertible matrix $(D_{v_j}Y_i)_{i,j=1}^d$, then, pathwise,
\[
 \operatorname{span}_{\R}
 \left\{\left(\int_Dv\dd\eta_i\right)_{i=1}^d:v\in\mathcal V\right\}
 \subsetneq\R^d,
\]
where $\eta_1,\ldots,\eta_d$ are the measures in
\eqref{eq:mixed-sources}.  Some $\alpha\in\R^d\setminus\{0\}$ would then satisfy
\[
 \int_Dv\dd\left(\sum_{i=1}^d\alpha_i\eta_i\right)=0,
 \qquad v\in\mathcal V.
\]
Uniform density in $C_0(D)$ would imply $\sum_i \alpha_i\eta_i=0$, contradicting
linear independence.  Therefore
\begin{equation}
 \PP\left(
 \bigcup_{\mathbf v\in\mathcal V^d}
 \left\{\det(D_{v_j}Y_i)_{i,j=1}^d\ne0\right\}
 \right)=1.
 \label{eq:countable-full-rank-cover}
\end{equation}

Fix a tuple $\mathbf v=(v_1,\ldots,v_d)$ with linearly independent
directions.  Let $w_1,\ldots,w_d$ be an energy-orthonormal basis of their
span and choose $M\in GL_d(\R)$ such that
\[
 v_j=\sum_{k=1}^dM_{kj}w_k,
 \qquad j=1,\ldots,d.
\]
Write
\[
 h=h^\perp+\sum_{k=1}^dt_kw_k,
 \qquad t\sim N(0,I_d),
\]
independently of $h^\perp$.  For almost every $h^\perp$,
Lemma~\ref{lem:coherent-finite-slice} gives a coherent family $(\mu_t)$.
Define the analytic extension
\[
 \widetilde Y_{h^\perp}(t)
 =\bigl(
   Z_1(h^\perp+\textstyle\sum_kt_kw_k),\ldots,
   Z_p(h^\perp+\textstyle\sum_kt_kw_k),
   \log\lambda_1(\mu_t),
   A_{f_1}(\mu_t),\ldots,A_{f_m}(\mu_t)
   \bigr).
\]
It agrees with the canonical vector $Y$ for $N(0,I_d)$-almost every $t$.
By the almost-sure identification of $D_vY$ with the analytic slice derivative
and Fubini, for almost every
$h^\perp$ and $N(0,I_d)$-almost every $t$,
\[
 (D_{v_j}Y_i)_{i,j=1}^d
 =D\widetilde Y_{h^\perp}(t)M.
\]
Thus the full-rank event for the fixed tuple agrees almost everywhere on the
slice with
\[
 O_{h^\perp}
 =\{t:\det D\widetilde Y_{h^\perp}(t)\ne0\}.
\]
The inverse function theorem and second countability cover $O_{h^\perp}$ by
countably many precompact open sets $U$ on which
$\widetilde Y_{h^\perp}$ is a $C^1$ diffeomorphism.  For every
Lebesgue-null $B\subset\R^d$, change of variables gives
\[
 \mathrm{Leb}_d\bigl(\widetilde Y_{h^\perp}^{-1}(B)\cap U\bigr)
 =\int_{B\cap\widetilde Y_{h^\perp}(U)}
   \left|\det D(\widetilde Y_{h^\perp}|_U)^{-1}(y)\right|\dd y=0.
\]
Summing over the countable cover yields
\[
 \mathrm{Leb}_d\bigl(
 \widetilde Y_{h^\perp}^{-1}(B)\cap O_{h^\perp}\bigr)=0.
\]
The Gaussian slice law has a strictly positive Lebesgue density, and
$Y=\widetilde Y_{h^\perp}$ almost everywhere on the slice.  Hence Fubini
gives
\[
 \PP\!\left(
 Y\in B,\ \det(D_{v_j}Y_i)_{i,j=1}^d\ne0
 \right)=0.
\]
Summing over the countable union in \eqref{eq:countable-full-rank-cover}
shows that $Y$ has a Lebesgue density.  The smooth map
$(z,\ell,y)\mapsto(z,e^\ell,y)$ preserves absolute continuity, proving
\eqref{eq:mixed-density}.
\end{proof}

\subsection{Conditional risk and design continuity}
\label{subsec:risk-continuity}

For a square-integrable real random variable $F$ and an observation $Y$ in a
standard Borel space, write
\[
 \Risk(F\mid Y)
 =\inf_b\EE|F-b(Y)|^2
 =\EE\Var(F\mid Y),
\]
where the infimum is over measurable real-valued predictors with finite
second moment.

\begin{lemma}
\label{lem:risk-continuity}
Let $F_n,F$ be real-valued random variables and let $Y_n,Y$ take values in a
common standard Borel space.  If
\[
 d_{\TV}(\Law(F_n,Y_n),\Law(F,Y))\longrightarrow0
\]
and $(F_n^2)$ is uniformly integrable, then
\[
 \Risk(F_n\mid Y_n)\longrightarrow\Risk(F\mid Y).
\]
\end{lemma}

\begin{proof}
Let $F_n^M=\clip_M(F_n)$ and $F^M=\clip_M(F)$, where
$\clip_M(x)=(-M)\vee(x\wedge M)$.  Predictors for the truncated targets may
be restricted to $[-M,M]$, so their squared losses are bounded by $4M^2$.
Therefore
\[
 |\Risk(F_n^M\mid Y_n)-\Risk(F^M\mid Y)|
 \le4M^2d_{\TV}(\Law(F_n,Y_n),\Law(F,Y))\longrightarrow0.
\]
Moreover,
\[
 |\Risk(F_n\mid Y_n)^{1/2}-\Risk(F_n^M\mid Y_n)^{1/2}|
 \le\|F_n-F_n^M\|_2,
\]
because distance to a closed subspace of $L^2$ is $1$-Lipschitz.  Uniform
integrability makes the right-hand side tend to zero uniformly in $n$ as
$M\to\infty$.  For every $R>0$, total-variation convergence gives
\[
 \EE(F^2\wedge R)
 =\lim_{n\to\infty}\EE(F_n^2\wedge R)
 \le\sup_n\EE F_n^2<\infty,
\]
so $F\in L^2$ and $\|F-F^M\|_2\to0$.  Letting first $n\to\infty$ and then
$M\to\infty$ proves the claim.
\end{proof}

The strict upper bound $\Delta_{\gamma,m,K}(g)<a$ uses the centered target
probe $\bar g$ from \eqref{eq:centered-target-probe}.

\begin{lemma}
\label{lem:positive-score-correlation}
For every $0<\gamma<2$,
\begin{equation}
 0<\EE[ZA_{\bar g,\gamma}]
 =\EE[\partial_gA_{\bar g,\gamma}]<\infty.
 \label{eq:positive-score-correlation}
\end{equation}
\end{lemma}

\begin{proof}
Decompose
\[
 h=h^\perp+t g,
 \qquad t=Z/a\sim N(0,a^{-1}),
\]
with $t$ independent of $h^\perp$.  Apply
Lemma~\ref{lem:coherent-finite-slice} to $w=g/\sqrt a$ and reparametrize its
coherent family by $t$.  Write this family as $(\mu_t)$, let
$\nu_t=\phi_1(t)^2\mu_t$, and set
\[
 A(t)=\int_D\bar g\dd\nu_t,\qquad
 p_a(t)=\sqrt{a/(2\pi)}e^{-at^2/2}.
\]
Then $|A(t)|\le\|\bar g\|_\infty$.  Let $\beta_{\bar g,t}$ be the measure in
Lemma~\ref{lem:response-measures} for the tilted medium at parameter $t$.
Then
\[
 A'(t)
 =\int_Dg\dd\beta_{\bar g,t}
 =\int_D\bar g\dd\beta_{\bar g,t}>0,
\]
where the second equality uses $\beta_{\bar g,t}(D)=0$ and the last one uses
\eqref{eq:strict-response}.  Integration by parts on $[-R,R]$ gives
\[
 \int_{-R}^RA'(t)p_a(t)\dd t
 =[A(t)p_a(t)]_{-R}^R
 +a\int_{-R}^RtA(t)p_a(t)\dd t.
\]
The boundary term tends to zero and
\[
 a\left|\int_{-R}^RtA(t)p_a(t)\dd t\right|
 \le a\|\bar g\|_\infty\int_\R|t|p_a(t)\dd t<\infty.
\]
Since $A'(t)p_a(t)\ge0$, monotone convergence yields
\[
 0<\int_\R A'(t)p_a(t)\dd t
 =a\int_\R tA(t)p_a(t)\dd t<\infty.
\]
Tonelli applies to the nonnegative left-hand side, while the preceding bound
permits Fubini on the right.  Averaging over $h^\perp$ and using $Z=at$
proves \eqref{eq:positive-score-correlation}.
\end{proof}

\subsection{Final proof of Theorem~\ref{thm:finite}}
\label{subsec:final-finite}

\begin{proof}
Fix a basis of $V_K$.  Lemma~\ref{lem:mixed-density}, applied to that basis
and the score direction $g$, gives a joint density $p(t,s)$ for
$(Z,S_{\gamma,V_K})$.  Put $p_S(s)=\int_\R p(t,s)\dd t$.  For
$p_S(s)\dd s$-almost every $s$,
\[
 \Law(Z\mid S_{\gamma,V_K}=s)(\dd t)
 =\frac{p(t,s)}{p_S(s)}\dd t.
\]
Since $\EE Z^2=a<\infty$, this conditional law has finite second moment for
almost every $s$.  Being absolutely continuous, it is not a point mass, so
\[
 \ell_\gamma(V_K;g)
 =\EE\Var(Z\mid S_{\gamma,V_K})>0.
\]
For every $V\subset V_K$, $S_{\gamma,V}$ is a deterministic postprocessing
of $S_{\gamma,V_K}$; hence
\begin{equation}
 I_\gamma(V;g)\le I_\gamma(V_K;g),\qquad
 \ell_\gamma(V;g)\ge\ell_\gamma(V_K;g).
 \label{eq:dictionary-domination}
\end{equation}

Fix a Euclidean inner product on $V_K$ and an orthonormal basis
$q_1,\ldots,q_K$, and set
\[
 A_K=(A_{q_1,\gamma},\ldots,A_{q_K,\gamma}).
\]
In a Grassmann coordinate neighborhood of $V$, choose continuously
$O(V)\in O(K)$ so that
\[
 q_i(V)=\sum_{k=1}^KO(V)_{ik}q_k,
 \qquad i=1,\ldots,m,
\]
is an orthonormal basis of $V$.  Thus the first $m$ coordinates of
$O(V)A_K$ are the corresponding response coordinates.  If $V_n\to V$ and
$p(t,\lambda,y)$ is the joint density of
$(Z,\lambda_\gamma,A_K)$, then the density after rotation is
\[
 p_{V_n}(t,\lambda,y)
 =p(t,\lambda,O(V_n)^\top y).
\]
Since $O(V_n)\to O(V)$,
\[
 \|p_{V_n}-p_V\|_{L^1}\longrightarrow0
\]
by approximation of $p$ in $L^1$ by compactly supported continuous
functions.  Hence
\[
 \Law(Z,\lambda_\gamma,O(V_n)A_K)
 \longrightarrow
 \Law(Z,\lambda_\gamma,O(V)A_K)
\]
in total variation.  Projection onto the first $m$ response coordinates
preserves total-variation convergence, and
Lemma~\ref{lem:risk-continuity}, with target $Z$, gives continuity of
$V\mapsto I_\gamma(V;g)$.  Compactness of $\Gr(m,V_K)$ gives an attained
minimum of the loss.  Together with \eqref{eq:dictionary-domination}, this
proves \eqref{eq:main-finite}.

Finally suppose $\bar g\in V_K$.  By
Lemma~\ref{lem:positive-score-correlation},
$\EE[ZA_{\bar g,\gamma}]>0$ and $\Var(A_{\bar g,\gamma})>0$.
Orthogonal projection in $L^2(P_0)$ gives
\[
 I_\gamma(\spanop(\bar g);g)
 \ge\frac{\EE[ZA_{\bar g,\gamma}]^2}
 {\Var(A_{\bar g,\gamma})}>0.
\]
Extend $\spanop(\bar g)$ to an $m$-dimensional subspace of $V_K$.
Information cannot decrease when the generated sigma-algebra is enlarged,
so the optimal information is positive and
$\Delta_{\gamma,m,K}(g)<a$.
\end{proof}

\section{Spectral expansions and the error operator}
\label{sec:weak-disorder}

The preceding section shows that every finite record loses information but does
not quantify the loss.  To do so as \(\gamma\downarrow0\), we first identify
the probe that matches the Gaussian score at first order and the operator
controlling the second-order residual.  We then expand the random Liouville
readouts in Wiener chaos.

\subsection{Density derivatives and the matching probe}

For a positive \(\varrho\in C^\infty(\overline D)\) near \(1\), let
\((\lambda(\varrho),\phi(\varrho))\) be the normalized first eigenpair
\begin{equation}
 L\phi(\varrho)=\lambda(\varrho)\varrho\phi(\varrho),
 \qquad
 \int_D\varrho\phi(\varrho)^2\dd x=1.
 \label{eq:density-eigenproblem}
\end{equation}
The principal eigenpair is analytic in \(\varrho\) near \(1\) by perturbation theory \cite{Kato}.  All density
derivatives below are evaluated at \(\varrho=1\); write
\[
 \lambda'_u
 =\left.\frac{\dd}{\dd t}\lambda(1+tu)\right|_{t=0},
 \qquad
 \phi'_u
 =\left.\frac{\dd}{\dd t}\phi(1+tu)\right|_{t=0}.
\]

\begin{lemma}
\label{lem:matching}
For \(u\in C^\infty(\overline D)\),
\begin{equation}
 \lambda'_u=-\lambda_0\ip{u}{\psi^2},
 \qquad
 \Phi u:=\phi'_u
 =\lambda_0R_0(u\psi)-\tfrac12\ip{u}{\psi^2}\psi.
 \label{eq:first-density-variation}
\end{equation}
For \(f\in\mathcal P\),
\[
 \mathrm dA_f(\dd x)[u\dd x]=\ip{\rho_f}{u},
 \qquad
 \rho_f=\psi^2f+2\lambda_0\psi R_0(f\psi).
\]
Moreover,
\begin{equation}
 \rho_f=\psi(I+2\lambda_0R_0)(f\psi),
 \qquad f\psi\perp\psi,
 \label{eq:linearization-factorization}
\end{equation}
so \(f\mapsto\rho_f\) is injective on \(\mathcal P\) and
\(\int_D\rho_f\dd x=0\).  The function \(f_g\) in
\eqref{eq:intro-matching} belongs to \(\mathcal P\) and satisfies
\[
 \rho_{f_g}=Lg.
\]
\end{lemma}

\begin{proof}
Differentiating \eqref{eq:density-eigenproblem} gives
\[
 (L-\lambda_0)\phi'_u
 =\lambda'_u\psi+\lambda_0u\psi,
 \qquad
 2\ip{\psi}{\phi'_u}+\ip{u}{\psi^2}=0.
\]
Pairing the first identity with \(\psi\), applying \(R_0\), and then using the
normalization identity yields \eqref{eq:first-density-variation}.  Since
\(\ip{f}{\psi^2}=0\),
\[
 \mathrm dA_f(\dd x)[u\dd x]
 =\int_Dfu\psi^2\dd x+2\int_Df\psi\,\Phi u\dd x
 =\ip{\psi^2f+2\lambda_0\psi R_0(f\psi)}{u}.
\]
This is \eqref{eq:linearization-factorization}.  If
\(L\psi_j=\lambda_j\psi_j\), \(j\ge2\), then
\[
 (I+2\lambda_0R_0)\psi_j
 =\frac{\lambda_j+\lambda_0}{\lambda_j-\lambda_0}\psi_j.
\]
Hence \(I+2\lambda_0R_0\) is positive and boundedly invertible on
\(\psi^\perp\), proving injectivity.  Also \(R_0\psi=0\) gives
\(\int_D\rho_f\dd x=0\).

Since $g\in C_c^\infty(D)$, $(Lg)/\psi\in C_c^\infty(D)$ and
\[
 \ip{(Lg)/\psi}{\psi}=\int_DLg\dd x=0.
\]
On \(\psi^\perp\),
\[
 (I+2\lambda_0R_0)^{-1}
 =(L-\lambda_0)(L+\lambda_0)^{-1}.
\]
Define
\[
 f_g
 =\frac{(L-\lambda_0)(L+\lambda_0)^{-1}((Lg)/\psi)}{\psi}.
\]
The numerator lies in \(\psi^\perp\), so
\[
 \int_D f_g\psi^2\dd x=0,
 \qquad
 \rho_{f_g}=Lg.
\]
It remains to verify that \(f_g\in C^\infty(\overline D)\).  The numerator satisfies
\[
 (L-\lambda_0)(L+\lambda_0)^{-1}((Lg)/\psi)
 =(Lg)/\psi-2\lambda_0(L+\lambda_0)^{-1}((Lg)/\psi),
\]
so it is smooth up to \(\partial D\) and vanishes there.  In boundary
normal coordinates, smooth Dirichlet functions factor by the boundary
distance, while Hopf's lemma gives a nonvanishing boundary factor for
\(\psi\).  Hence the quotient is smooth up to \(\partial D\), and therefore
\(f_g\in\mathcal P\).
\end{proof}

To identify the second-order residual, put
\[
 u_f=R_0(f\psi),\qquad
 a_f=f\psi+\lambda_0u_f,\qquad
 b_f=\psi u_f+\Phi^*u_f.
\]

\begin{lemma}
\label{lem:density-hessian}
For \(f\in\mathcal P\), the density Hessian is represented on \(H\) by the
self-adjoint operator
\begin{equation}
 C_f
 =2\Phi^*M_f\Phi
 +2(M_{a_f}\Phi+\Phi^*M_{a_f})
 -2\lambda_0(\psi^2\otimes b_f+b_f\otimes \psi^2)
 \in \mathfrak S_2(H).
 \label{eq:explicit-density-hessian}
\end{equation}
For smooth \(u,v\),
\[
 \mathrm d^2A_f(\dd x)[u\dd x,v\dd x]=\ip{u}{C_fv}.
\]
\end{lemma}

\begin{proof}
For \(\varrho=1+su+tv\), write
\[
 \phi''_{uv}=\left.\partial_s\partial_t\phi(\varrho)\right|_{s=t=0},
 \qquad
 \lambda''_{uv}=\left.\partial_s\partial_t\lambda(\varrho)\right|_{s=t=0}.
\]
Twice differentiating \eqref{eq:density-eigenproblem} gives
\begin{align}
 (L-\lambda_0)\phi''_{uv}
 ={}&\lambda''_{uv}\psi
 +\lambda'_u(v\psi+\Phi v)
 +\lambda'_v(u\psi+\Phi u)\notag\\
 &+\lambda_0(u\Phi v+v\Phi u).
 \label{eq:second-eigenfunction-equation}
\end{align}
Apply \(R_0=(L-\lambda_0)^{-1}Q\) to
\eqref{eq:second-eigenfunction-equation} and pair with \(f\psi\).
Since \(f\psi\perp\psi\) and \(R_0\) is self-adjoint,
\begin{align*}
 \ip{f\psi}{\phi''_{uv}}
 ={}&-\lambda_0\ip{\psi^2}{u}
       \ip{u_f}{v\psi+\Phi v}
     -\lambda_0\ip{\psi^2}{v}
       \ip{u_f}{u\psi+\Phi u}\\
 &+\lambda_0\ip{u_f}{u\Phi v+v\Phi u}.
\end{align*}
Direct differentiation of \(A_f(\varrho\dd x)=\int_Df\varrho\phi(\varrho)^2\dd x\) yields
\[
 \mathrm d^2A_f(\dd x)[u\dd x,v\dd x]
 =2\ip{\Phi u}{M_f\Phi v}
  +2\ip{u}{M_{f\psi}\Phi v}
  +2\ip{v}{M_{f\psi}\Phi u}
  +2\ip{f\psi}{\phi''_{uv}}.
\]
Substitution of the preceding identity and
\(\lambda'_u=-\lambda_0\ip{\psi^2}{u}\) gives
\begin{align*}
 \mathrm d^2A_f(\dd x)[u\dd x,v\dd x]
 ={}&2\ip{\Phi u}{M_f\Phi v}
 +2\ip{u}{M_{a_f}\Phi v}
 +2\ip{v}{M_{a_f}\Phi u}\\
 &-2\lambda_0\{
   \ip{\psi^2}{u}\ip{b_f}{v}
  +\ip{\psi^2}{v}\ip{b_f}{u}\},
\end{align*}
which is \eqref{eq:explicit-density-hessian}.

In dimension two \(R_0\in \mathfrak S_2(H)\), and
\[
 \Phi=\lambda_0R_0M_\psi-\tfrac12\psi\otimes\psi^2\in \mathfrak S_2(H).
\]
Elliptic regularity gives \(a_f,b_f\in C^\infty(\overline D)\).  Hence
\[
 \Phi^*M_f\Phi\in \mathfrak S_1,\qquad
 M_{a_f}\Phi+\Phi^*M_{a_f}\in \mathfrak S_2,
\]
while the last term of \eqref{eq:explicit-density-hessian} has finite rank.
Thus \(C_f\in \mathfrak S_2(H)\).
\end{proof}

\subsection{The second-order error operator}

The matching probe captures the Gaussian score at first order.  We now isolate
the second-order residual and its spectral tail.

For \(f\in\mathcal P\), set
\begin{equation}
 J_f=\tfrac12G^{1/2}(M_{\rho_f}+C_f)G^{1/2},
 \qquad
 d_f=\tfrac12\Tr(G^{1/2}C_fG^{1/2}),
 \qquad
 \xi_f=I_2(J_f)+d_f.
 \label{eq:general-error-kernel}
\end{equation}
Then \(J_{f_g}=K_g\).

\begin{proposition}
\label{prop:weighted-weyl}
For real \(\rho\in C_c^\infty(D)\),
\[
 B_\rho=\tfrac12G^{1/2}M_\rho G^{1/2}
\]
satisfies
\begin{equation}
 \lim_{j\to\infty}j\,s_j(B_\rho)
 =\frac14\int_D|\rho(x)|\dd x.
 \label{eq:weighted-weyl}
\end{equation}
\end{proposition}

\begin{proof}
Choose real \(\chi,\chi_1\in C_c^\infty(D)\) with
\(\chi=1\) near \(\supp\rho\) and \(\chi_1=1\) near \(\supp\chi\).  Then
\[
 B_\rho
 =\tfrac12(\chi G^{1/2})^*M_\rho(\chi G^{1/2}),
 \qquad
 \chi G^{1/2}
 =\chi G^{1/2}\chi_1+\chi G^{1/2}(1-\chi_1).
\]
Separated-support regularity gives, for every \(s\ge0\),
\[
 \|\chi G^{1/2}(1-\chi_1)u\|_{H^s(D)}
 \le C_s\|u\|_{L^2(D)}.
\]
Because the range has fixed compact support, the singular values of the
Sobolev embedding give
\[
 s_j\!\left(\chi G^{1/2}(1-\chi_1)\right)
 \le C_s j^{-s/2}.
\]
Hence this term belongs to \(\mathfrak S_p\) for every \(p>0\).  With
\(A=\chi G^{1/2}\chi_1\),
\[
 B_\rho-\tfrac12A^*M_\rho A
 \in\bigcap_{p>0}\mathfrak S_p.
\]
The interior localization of Dirichlet powers
\cite[Section~2, arXiv version]{RozenblumShargorodsky} gives, after extension
by zero,
\[
 A=A_0+S,
\]
where \(A_0\) is a compactly supported pseudodifferential operator of order
\(-1\) with principal symbol \(\sqrt{2\pi}\,\chi(x)|\xi|^{-1}\), while
\(S\) is smoothing.  Hence
\[
 \tfrac12A^*M_\rho A-\tfrac12A_0^*M_\rho A_0
 \in\bigcap_{p>0}\mathfrak S_p.
\]
Theorem~6.12 of \cite{Ponge} applies to
\(\tfrac12A_0^*M_\rho A_0\), whose principal symbol is
\(\pi\rho(x)|\xi|^{-2}\).  Its Weyl coefficient is
\[
 (2\pi)^{-2}\int_D
 \bigl|\{\xi\in\R^2:\pi|\rho(x)||\xi|^{-2}>1\}\bigr|\dd x
 =\frac14\|\rho\|_{L^1(D)}.
\]
Choose \(0<p<1\).  The preceding remainder then satisfies
\[
 j\,s_j\!\left(B_\rho-\tfrac12A^*M_\rho A\right)\longrightarrow0,
 \qquad
 t\,n\!\left(t;B_\rho-\tfrac12A^*M_\rho A\right)\longrightarrow0.
\]
The singular-value counting inequalities therefore preserve the preceding
Weyl constant, proving \eqref{eq:weighted-weyl}.
\end{proof}

\begin{lemma}
\label{lem:error-tail}
The quantities in \eqref{eq:general-error-kernel} are well defined, and
\(J_f\in \mathfrak S_2(H)\) is self-adjoint.  For the matching probe,
\begin{equation*}
 K_g-\tfrac12G^{1/2}M_{Lg}G^{1/2}\in \mathfrak S_1,
 \qquad
 s_j(K_g)\sim\frac{c_g}{j},
 \qquad
 c_g=\tfrac14\|Lg\|_{L^1(D)}>0.
\end{equation*}
In particular, \(K_g\) has infinite rank.
\end{lemma}

\begin{proof}
The two-dimensional Dirichlet Weyl law gives \(G^{1/2}\in \mathfrak S_4(H)\).
Since \(M_{\rho_f}\) is bounded and \(C_f\in \mathfrak S_2(H)\), Schatten
H\"older gives
\[
 G^{1/2}M_{\rho_f}G^{1/2}\in \mathfrak S_2,
 \qquad
 G^{1/2}C_fG^{1/2}\in \mathfrak S_1.
\]
Thus \(J_f\in \mathfrak S_2\) and \(d_f\) is finite.  Since \(M_{\rho_f}\) and \(C_f\) are self-adjoint, so is \(J_f\).  Since \(\rho_{f_g}=Lg\),
\[
 K_g-\tfrac12G^{1/2}M_{Lg}G^{1/2}
 =\tfrac12G^{1/2}C_{f_g}G^{1/2}\in \mathfrak S_1.
\]

Set \(B=\frac12G^{1/2}M_{Lg}G^{1/2}\).  Proposition~\ref{prop:weighted-weyl} gives
\[
 j\,s_j(B)\longrightarrow c_g.
\]
Since \(K_g-B\in \mathfrak S_1\),
\[
 \frac j2s_j(K_g-B)
 \le\sum_{k=\lfloor j/2\rfloor}^{j}s_k(K_g-B)\longrightarrow0.
\]
Thus \(t\,n(t;K_g-B)\to0\), where
\(n(t;A)=\#\{j:s_j(A)>t\}\).  For \(0<\varepsilon<1\),
\[
 \begin{aligned}
 n(t;K_g)
 &\le n((1-\varepsilon)t;B)+n(\varepsilon t;K_g-B),\\
 n((1+\varepsilon)t;B)
 &\le n(t;K_g)+n(\varepsilon t;K_g-B).
 \end{aligned}
\]
Since \(t\,n(t;B)\to c_g\),
\[
 \frac{c_g}{1+\varepsilon}
 \le\liminf_{t\downarrow0}t\,n(t;K_g)
 \le\limsup_{t\downarrow0}t\,n(t;K_g)
 \le\frac{c_g}{1-\varepsilon}.
\]
Letting \(\varepsilon\downarrow0\) and inverting the counting asymptotic gives
\[
 t\,n(t;K_g)\longrightarrow c_g,
 \qquad
 j\,s_j(K_g)\longrightarrow c_g.
\]
Finally \(Lg\ne0\) because \(g\ne0\) and the Dirichlet form is positive
definite, so \(c_g>0\).
\end{proof}

\subsection{Energy operators}

To transfer the deterministic density expansion above to the random spectral
readouts, we realize the varying speed-measure problem on the fixed energy
space \(H_E\).

Fix a finite positive smooth measure \(\mu\) for which \(G_\mu\) is a positive
Hilbert--Schmidt Green inverse and \(u\mu\) has finite energy for every
\(u\in L^2(\mu)\).  These properties hold almost surely for
\(\mu=\mu_\gamma\) by Section~\ref{subsec:green-input}, and also for
Lebesgue measure.  Initially let
\[
 \iota_\mu:C_c^\infty(D)\subset H_E\longrightarrow L^2(\mu),
 \qquad
 \iota_\mu u=u.
\]

\begin{lemma}
The map \(\iota_\mu\) extends to a compact operator
\(H_E\to L^2(\mu)\), and
\[
 \iota_\mu^*b=G(b\mu),
 \qquad
 \iota_\mu \iota_\mu^*=G_\mu,
 \qquad
 T_\mu:=\iota_\mu^*\iota_\mu\in \mathfrak S_2(H_E).
\]
The nonzero eigenvalues of \(T_\mu\) are
\(\lambda_j(\mu)^{-1}\), counted with multiplicity.
\end{lemma}

\begin{proof}
For \(u\in C_c^\infty(D)\),
\begin{align*}
 \|u\|_{L^2(\mu)}^2
 &=\sup_{\|b\|_{L^2(\mu)}\le1}
    |E(u,G(b\mu))|^2\\
 &\le E(u,u)
    \sup_{\|b\|_{L^2(\mu)}\le1}
      \langle b,G_\mu b\rangle_{L^2(\mu)}
 \le\|G_\mu\|\,E(u,u).
\end{align*}
Thus \(\iota_\mu\) extends continuously.  Since \(\mu\) is smooth, the extension
is the quasi-continuous trace.  Finite-energy duality gives
\[
 E(\iota_\mu^*b,u)
 =\langle b,\iota_\mu u\rangle_{L^2(\mu)}
 =E(G(b\mu),u),
\]
hence \(\iota_\mu^*b=G(b\mu)\) and
\(\iota_\mu \iota_\mu^*=G_\mu\).  Compactness of \(G_\mu\) implies compactness of
\(\iota_\mu\).  The nonzero spectra of \(\iota_\mu \iota_\mu^*\) and
\(\iota_\mu^*\iota_\mu\) agree, and
\[
 \|T_\mu\|_{\HS}^2
 =\sum_{j\ge1}\lambda_j(\mu)^{-2}
 =\|G_\mu\|_{\HS}^2<\infty.
\]
\end{proof}

For \(0\le\gamma<2\), set \(\mu_0=\dd x\) and, for bounded smooth \(f\),
\[
 T_\gamma(f)=\iota_{\mu_\gamma}^*M_f\iota_{\mu_\gamma},
 \qquad
 T_\gamma=T_\gamma(1).
\]
Since \(T_\gamma=\iota_{\mu_\gamma}^*\iota_{\mu_\gamma}\in\mathfrak S_2\),
\(\iota_{\mu_\gamma}\in\mathfrak S_4\), and hence
\[
 \|T_\gamma(f)\|_{\mathfrak S_2}
 \le \|f\|_\infty\|\iota_{\mu_\gamma}\|_{\mathfrak S_4}^2<\infty.
\]
Moreover,
\[
 E(u,T_\gamma(f)v)
 =\int_Df\,\widetilde u\,\widetilde v\dd\mu_\gamma,
 \qquad u,v\in H_E.
\]
The top eigenvalue of \(T_\gamma\) is
\(r_\gamma=\lambda_\gamma^{-1}\).  Under the unitary map \(G^{1/2}:H\to H_E\), the operators
\(T_0(f)\) and \(T_0\) are represented by
\[
 G^{1/2}M_fG^{1/2}
 \qquad\text{and}\qquad
 G,
\]
respectively.

\subsection{Wick expansion of the energy operator}

On the fixed energy space, the Liouville weight can now be expanded directly in Wiener chaos.  For \(\varepsilon>0\), set
\[
 \mu_{\gamma,\varepsilon}(\dd x)
 =\exp\!\left(
   \gamma h_\varepsilon(x)
   -\frac{\gamma^2}{2}\EE h_\varepsilon(x)^2
 \right)\dd x,
 \qquad
 T_{\gamma,\varepsilon}(f)
 =\iota_{\mu_{\gamma,\varepsilon}}^*M_f\iota_{\mu_{\gamma,\varepsilon}}.
\]
For \(k\ge0\), define the \(k\)th Wick coefficient first at cutoff
\(\varepsilon>0\) by
\[
 E(u,B_{k,\varepsilon}(f)v)
 =\int_Dfuv:h_\varepsilon^k:\dd x,
 \qquad
 u,v\in C_c^\infty(D).
\]

\begin{lemma}
\label{lem:operator-expansion}
Let \(f\in C^\infty(\overline D)\).  For integers \(r,N\ge0\) and finite \(p\ge2\), there exist
\(\gamma_*>0\) and \(C<\infty\) such that, for
\(0<\gamma<\gamma_*\),
\begin{equation}
 \left\|
 T_\gamma(f)-\sum_{k=0}^N\frac{\gamma^k}{k!}B_k(f)
 \right\|_{\Dsob^{r,p}(\mathfrak S_2(H_E))}
 \le C\|f\|_\infty\gamma^{N+1}.
 \label{eq:operator-expansion}
\end{equation}
Here \(B_k(f)\) is the \(L^2(\mathfrak S_2(H_E))\)-limit of
\(B_{k,\varepsilon}(f)\), and
\begin{equation}
 \EE\|B_k(1)\|_{\HS}^2
 =k!\iint_{D^2}G(x,y)^{k+2}\dd x\dd y.
 \label{eq:wick-hs-isometry}
\end{equation}
\end{lemma}

\begin{proof}
Under \(G^{1/2}:H\to H_E\), the operator
\(B_{k,\varepsilon}(f)\) is represented by
\[
 G^{1/2}M_{f:h_\varepsilon^k:}G^{1/2}.
\]
If \(G_t=e^{-tL}G\), Wick covariance and the Hilbert--Schmidt kernel
identity give
\[
 \EE\langle B_{k,\varepsilon}(f),B_{k,\delta}(f)\rangle_{\HS}
 =k!\iint_{D^2}
 f(x)f(y)G(x,y)^2G_{\varepsilon+\delta}(x,y)^k
 \dd x\dd y.
\]
Since \(0\le G_t\le G\), dominated convergence gives
\[
 \begin{aligned}
 \EE\|B_{k,\varepsilon}(f)-B_{k,\delta}(f)\|_{\HS}^2
 &=\EE\|B_{k,\varepsilon}(f)\|_{\HS}^2
   +\EE\|B_{k,\delta}(f)\|_{\HS}^2\\
 &\quad
   -2\EE\langle B_{k,\varepsilon}(f),
                B_{k,\delta}(f)\rangle_{\HS}
 \longrightarrow0.
 \end{aligned}
\]
Denote the limit by \(B_k(f)\).  The logarithmic Green bound yields
\[
 \int_0^1(1+\log(1/r))^{k+2}r\dd r
 =\int_0^\infty(1+t)^{k+2}e^{-2t}\dd t
 \le C^{k+3}(k+2)!,
\]
and hence
\[
 \EE\|B_k(f)\|_{\HS}^2
 \le k!\|f\|_\infty^2
      \iint_{D^2}G(x,y)^{k+2}\dd x\dd y
 \le C^{k+3}k!(k+2)!\|f\|_\infty^2.
\]
For \(f=1\) the first inequality is an equality, giving
\eqref{eq:wick-hs-isometry}.

Each \(B_k(f)\) lies in the \(k\)th Wiener chaos.  By
\eqref{eq:hypercontractivity}, for fixed \(r,p\),
\[
 \frac{
 \|B_{k,\varepsilon}(f)\|_{\Dsob^{r,p}(\mathfrak S_2)}
 +\|B_k(f)\|_{\Dsob^{r,p}(\mathfrak S_2)}
 }{k!}
 \le C_{r,p}^{\,k+1}\|f\|_\infty(1+k)^{1+r/2}.
\]
Hence, for sufficiently small \(\gamma\),
\[
 \sum_{k\ge0}\frac{\gamma^k}{k!}B_{k,\varepsilon}(f),
 \qquad
 \sum_{k\ge0}\frac{\gamma^k}{k!}B_k(f)
\]
converge in \(\Dsob^{r,p}(\mathfrak S_2(H_E))\), and the tails after order \(N\)
are \(O(\|f\|_\infty\gamma^{N+1})\), uniformly in \(\varepsilon\).

For fixed \(\varepsilon\), Wick's exponential identity gives
\[
 T_{\gamma,\varepsilon}(f)
 =\sum_{k\ge0}\frac{\gamma^k}{k!}B_{k,\varepsilon}(f).
\]
Letting \(\varepsilon\downarrow0\) in the series gives an
\(\mathfrak S_2(H_E)\)-valued limit with the same tail estimate.  On a countable
energy-dense family \(u_i\in C_c^\infty(D)\),
\[
 E(u_i,T_{\gamma,\varepsilon}(f)u_j)
 =\int_Dfu_iu_j\dd\mu_{\gamma,\varepsilon}
 \longrightarrow
 \int_Dfu_iu_j\dd\mu_\gamma
 =E(u_i,T_\gamma(f)u_j)
\]
in probability.  Uniqueness of limits and continuity of both bilinear forms
identify the operator-series limit with \(T_\gamma(f)\).  This proves
\eqref{eq:operator-expansion}.
\end{proof}

\subsection{Expansion of the spectral readouts}

We now pass the operator expansion through the principal-eigenvalue and
response maps.  Let
\[
 r_0=\lambda_0^{-1},
 \qquad
 c_0=\frac{r_0}{(1+r_0)^2},
\]
and define
\begin{equation}
 Y_{\gamma,0}
 =\frac{\frac{r_\gamma}{1+r_\gamma}-\frac{r_0}{1+r_0}}{\gamma c_0},
 \qquad
 Y_{\gamma,f}
 =\frac{A_{f,\gamma}}{\gamma},
 \quad f\in\mathcal P.
 \label{eq:normalized-observations}
\end{equation}
For every \(\gamma>0\), these are invertible deterministic
transformations of the corresponding scalar readouts.

\begin{lemma}
\label{lem:observation-expansion}
Fix a finite-dimensional \(\mathcal F\subset\mathcal P\), integers \(r,M,N\ge0\),
and finite \(p\ge2\).  For sufficiently small \(\gamma>0\), uniformly for
\(f\) in bounded subsets of \(\mathcal F\),
\begin{align}
 Y_{\gamma,0}
 &=W(v_0)+\gamma\xi_0+O_{L^p}(\gamma^2),\notag\\
 Y_{\gamma,f}
 &=W(v_f)+\gamma\xi_f+O_{L^p}(\gamma^2),
 \label{eq:readout-expansion}
\end{align}
Here \(\xi_0\) denotes the second-order coefficient of \(Y_{\gamma,0}\)
and belongs to the sum of the zeroth and second Wiener chaoses; \(\xi_f\) is
defined in \eqref{eq:general-error-kernel}.  There are
localized versions with the same expansion in \(\Dsob^{r,p}\), and
\begin{equation}
 \PP\{(Y_{\gamma,0},Y_{\gamma,f})
 \ne(\widetilde Y_{\gamma,0},\widetilde Y_{\gamma,f})\}
 =O(\gamma^M).
 \label{eq:localization-probability}
\end{equation}
Moreover,
\begin{equation}
 \|Y_{\gamma,f}-\widetilde Y_{\gamma,f}\|_{L^p}
 +\|Y_{\gamma,0}-\widetilde Y_{\gamma,0}\|_{L^p}
 =O(\gamma^N).
 \label{eq:localization-lp-error}
\end{equation}
\end{lemma}

\begin{proof}
The deterministic operator \(T_0\) has a simple isolated top eigenvalue
\(r_0\).  Choose \(\delta>0\) so that on
\(\|T-T_0\|_{\HS}<2\delta\) the top eigenvalue \(r(T)\) and spectral
projection \(\Pi(T)\) are analytic and \(r(T)>r_0/2\).  If \(\Gamma\)
encloses \(r_0\) and no other point of \(\operatorname{Spec}(T_0)\), then
\[
 \Pi(T)=\frac{1}{2\pi i}
 \int_\Gamma(\zeta I-T)^{-1}\dd\zeta.
\]
The resolvent identity gives uniform bounds for every fixed derivative of \(\Pi(T)\) and \(r(T)=\Tr(T\Pi(T))\) on
\(\|T-T_0\|_{\HS}\le3\delta/2\).

Choose a smooth cutoff \(\chi(T)\), equal to \(1\) on the
\(\delta\)-ball and \(0\) outside the \(3\delta/2\)-ball, and set locally
\[
 \mathcal A(T,B)
 =\chi(T)\frac{\Tr(\Pi(T)B)}{r(T)}.
\]
Extend this expression by \(0\) outside the analytic ball.  The resulting
map is smooth, linear in \(B\), and for every fixed \(j\),
\[
 \|\mathrm d^j\mathcal A(T,B)\|
 \le C_j(1+\|B\|_{\HS}).
\]
Use the same cutoff for
\[
 \frac{r_0}{1+r_0}
 +\chi(T)\left\{\frac{r(T)}{1+r(T)}-\frac{r_0}{1+r_0}\right\}.
\]
On the event \(\|T_\gamma-T_0\|_{\HS}\le\delta\), let
\[
 e_\gamma=\frac{\phi_\gamma}{\sqrt{\lambda_\gamma}},
 \qquad \|e_\gamma\|_E=1.
\]
Since \(\Pi(T_\gamma)=e_\gamma\otimes e_\gamma\) on \(H_E\),
\[
 \Tr\!\bigl(\Pi(T_\gamma)T_\gamma(f)\bigr)
 =E(e_\gamma,T_\gamma(f)e_\gamma)
 =r_\gamma A_{f,\gamma}.
\]
Thus \(\mathcal A(T_\gamma,T_\gamma(f))=A_{f,\gamma}\) on this event.
Define the localized response from \(\mathcal A(T_\gamma,T_\gamma(f))\) and
the localized eigenvalue readout from the preceding scalar map, both
normalized as in \eqref{eq:normalized-observations}.  They equal the actual
readouts on \(\|T_\gamma-T_0\|_{\HS}\le\delta\).

Regard $\mathcal A$ as a smooth map on the product of two
self-adjoint Hilbert--Schmidt spaces and write
\[
 \mathbf T_{\gamma,f}:=(T_\gamma,T_\gamma(f)),
 \qquad
 \mathbf T_{0,f}:=(T_0,T_0(f)).
\]
Since $\mathcal A(\mathbf T_{0,f})=0$, Taylor's formula and
Lemma~\ref{lem:operator-expansion} give
\begin{align*}
 &\Bigl\|
 \mathcal A(\mathbf T_{\gamma,f})
 -\mathrm d\mathcal A(\mathbf T_{0,f})
       [\mathbf T_{\gamma,f}-\mathbf T_{0,f}]\\
 &\qquad
 -\tfrac12\mathrm d^2\mathcal A(\mathbf T_{0,f})
       [\mathbf T_{\gamma,f}-\mathbf T_{0,f},
        \mathbf T_{\gamma,f}-\mathbf T_{0,f}]
 \Bigr\|_{\Dsob^{r,p}}
 \le C\gamma^3.
\end{align*}
The bound follows from Taylor's integral remainder together with the
Malliavin chain and product rules and the \(O(\gamma)\) increment estimate
in \eqref{eq:operator-expansion}.

To identify the first two coefficients, fix a field cutoff
\(\varepsilon>0\).  Since the localization cutoff equals one near \(T_0\),
its first two Taylor coefficients agree with those of the exact regularized
response.  Density Taylor expansion gives
\begin{equation*}
 A_f(\mu_{\gamma,\varepsilon})
 =\gamma\ip{\rho_f}{h_\varepsilon}
 +\frac{\gamma^2}{2}
 \left(
   \ip{\rho_f}{:h_\varepsilon^2:}
   +\ip{h_\varepsilon}{C_fh_\varepsilon}
 \right)
 +O_{\varepsilon,h}(\gamma^3).
\end{equation*}
The linear and quadratic coefficients are
\[
 W(e^{-\varepsilon L}v_f),
 \qquad
 I_2(e^{-\varepsilon L}J_fe^{-\varepsilon L})
 +\frac12\Tr(
 e^{-\varepsilon L}G^{1/2}C_fG^{1/2}e^{-\varepsilon L}).
\]
Since \(J_f\in \mathfrak S_2\) and \(G^{1/2}C_fG^{1/2}\in \mathfrak S_1\),
\[
 e^{-\varepsilon L}J_fe^{-\varepsilon L}\to J_f
 \quad\text{in }\mathfrak S_2,
 \qquad
 e^{-\varepsilon L}G^{1/2}C_fG^{1/2}e^{-\varepsilon L}
 \to G^{1/2}C_fG^{1/2}
 \quad\text{in }\mathfrak S_1.
\]
Equations \eqref{eq:second-chaos-normalization} and
\eqref{eq:hypercontractivity} therefore identify the limiting coefficients
as \(W(v_f)\) and \(\xi_f\).

For the eigenvalue coordinate, the normalized top eigenvector of \(T_0\) in
\(H_E\) is \(\psi/\sqrt{\lambda_0}\), so
\[
 \mathrm dr(T_0)[B_1(1)]
 =E\!\left(
   \frac{\psi}{\sqrt{\lambda_0}},
   B_1(1)\frac{\psi}{\sqrt{\lambda_0}}
 \right)
 =r_0W(v_0).
\]
Since the derivative factor at \(r_0\) equals
\(r_0/(1+r_0)^2=c_0\), the linear coefficient of
\(Y_{\gamma,0}\) is \(W(v_0)\); its quadratic coefficient has only zeroth
and second Wiener chaos components.  This proves the localized expansions
in \eqref{eq:readout-expansion}.

Finally, Lemma~\ref{lem:operator-expansion} with \(N=0\) gives, for every
finite \(q\ge2\),
\[
 \|T_\gamma-T_0\|_{L^q(\mathfrak S_2)}\le C_q\gamma,
 \qquad
 \PP(\|T_\gamma-T_0\|_{\HS}>\delta)
 \le \delta^{-q}C_q^q\gamma^q.
\]
On the exceptional event,
\[
 |A_{f,\gamma}|\le\|f\|_\infty,\qquad
 0<\frac{r_\gamma}{1+r_\gamma}<1,\qquad
 |\mathcal A(T,B)|\le C\|B\|_{\HS}.
\]
H\"older's inequality therefore gives
\[
 \|Y_{\gamma,f}-\widetilde Y_{\gamma,f}\|_{L^p}
 \le C\gamma^{-1}
 (1+\|T_\gamma(f)\|_{L^{2p}(\mathfrak S_2)})
 \PP(\|T_\gamma-T_0\|_{\HS}>\delta)^{1/(2p)},
\]
and the same estimate holds for the eigenvalue coordinate.  Choose
\[
 q\ge\max\{M,2p(\max\{N,2\}+1)\}.
\]
Then the probability bound above gives \eqref{eq:localization-probability},
while H\"older's estimate gives \eqref{eq:localization-lp-error}.  Combining
this with the localized expansions proves the \(L^p\) expansions for the
actual readouts.
\end{proof}

\section{Conditional risk and optimal designs}
We convert the weak-disorder expansions of
Section~\ref{sec:weak-disorder} into conditional risks and then optimize over
finite probe designs.  Because nonzero rescaling does not change a generated
sigma-algebra, the required limits are obtained from normalized residuals in
total variation rather than from the readouts alone in $L^2$.

\subsection{Convergence of conditional risks}

\begin{lemma}
\label{lem:tv-convergence}
For every $d\ge1$ there exists a finite $p\ge2$ such that, whenever
$F_n,F\in\Dsob^{3,p}(\R^d)$,
\[
 F_n\longrightarrow F\quad\text{in }\Dsob^{3,p}(\R^d),
 \qquad
 \det\sigma_F>0\quad\text{a.s.},
\]
one has
\begin{equation}
 d_{\TV}(\Law(F_n),\Law(F))\longrightarrow0.
 \label{eq:sobolev-tv-convergence}
\end{equation}
\end{lemma}

\begin{proof}
By \cite[Theorem~2.19]{BallyCaramellino}, after approximation by smooth
cylindrical variables, there are $p\ge2$, $b>0$ and $C<\infty$, depending
only on $d$, such that
\begin{align}
 d_{\TV}(\Law(F),\Law(F'))
 &\le C\varepsilon^{-b}
   (1+\norm{F}_{3,p}+\norm{F'}_{3,p})^b
   (\EE|F-F'|)^{1/2}\notag\\
 &\quad+C\PP(\det\sigma_F<3\varepsilon)
       +C\PP(\det\sigma_{F'}<3\varepsilon),
 \qquad 0<\varepsilon<1.
 \label{eq:bc-estimate}
\end{align}
Indeed,
\[
 \|\sigma_{F^{(j)}}-\sigma_F\|
 \le C_d\bigl(\|\mathrm DF^{(j)}\|+\|\mathrm DF\|\bigr)
          \|\mathrm D(F^{(j)}-F)\|,
\]
so convergence in \(\Dsob^{3,p}\) implies
\(\sigma_{F^{(j)}}\to\sigma_F\) in probability.  The corresponding
cylindrical approximants of \(F'\) satisfy
\(\sigma_{F'^{(j)}}\to\sigma_{F'}\) in probability as well.  Total variation
is lower semicontinuous under weak convergence.

Apply \eqref{eq:bc-estimate} to $F_n,F$.  Their Sobolev norms are bounded,
$\EE|F_n-F|\to0$, and $\sigma_{F_n}\to\sigma_F$ in probability.  Continuity
of the determinant gives
\[
 \limsup_{n\to\infty}\PP(\det\sigma_{F_n}<3\varepsilon)
 \le \PP(\det\sigma_F<9\varepsilon).
\]
Therefore
\[
 \limsup_{n\to\infty}d_{\TV}(\Law(F_n),\Law(F))
 \le C\PP(\det\sigma_F<9\varepsilon).
\]
Letting $\varepsilon\downarrow0$ proves
\eqref{eq:sobolev-tv-convergence}.
\end{proof}

\begin{lemma}
\label{lem:chaos-nondegeneracy}
Let $B$ be an infinite-rank self-adjoint Hilbert--Schmidt operator on $H$,
let $e_1,\ldots,e_d$ be orthonormal, and let $v\in H$, $c\in\R$.  Then the
Malliavin covariance matrix of
\[
 (I_2(B)+W(v)+c,W(e_1),\ldots,W(e_d))
\]
is positive definite almost surely.
\end{lemma}

\begin{proof}
Put $U=\spanop(e_1,\ldots,e_d)$.  For an orthonormal basis $(b_j)$ of $H$,
\[
 BW:=\sum_{j\ge1}W(b_j)Bb_j
 \quad\text{in }L^2(P_0;H),
 \qquad
 \EE\norm{BW}^2=\norm{B}_{\HS}^2,
\]
and
\[
 \mathrm D(I_2(B)+W(v)+c)=2BW+v.
\]
The Schur complement of the identity block in the Malliavin covariance
matrix equals
\[
 \norm{P_{U^\perp}(2BW+v)}^2.
\]
If $P_{U^\perp}B=0$, then $B$ has finite rank.  Hence there exists
$u\in U^\perp$ with $Bu\ne0$, and
\[
 \ip{u}{2BW+v}=2W(Bu)+\ip{u}{v}
\]
is a nondegenerate Gaussian variable.  The Schur complement is therefore
positive almost surely.
\end{proof}

\begin{lemma}
\label{lem:localized-risk}
Let $F_n,F$ be real-valued and let $Y_n,Y$ take values in $\R^d$.  Let
$(\widetilde F_n,\widetilde Y_n)$ be defined on the same Gaussian space, and
choose $p$ as in Lemma~\ref{lem:tv-convergence} for dimension $d+1$.  If
\begin{align*}
 &\PP\bigl((F_n,Y_n)\ne(\widetilde F_n,\widetilde Y_n)\bigr)\to0,\\
 &(\widetilde F_n,\widetilde Y_n)\to(F,Y)
       \quad\text{in }\Dsob^{3,p}(\R^{d+1}),\\
 &\det\sigma_{(F,Y)}>0\quad\text{a.s.},
 \qquad
 \sup_n\EE|F_n|^{2+\eta}<\infty
 \quad\text{for some }\eta>0,
\end{align*}
then
\begin{equation}
 \Risk(F_n\mid Y_n)\longrightarrow\Risk(F\mid Y).
 \label{eq:localized-risk}
\end{equation}
\end{lemma}

\begin{proof}
The coupling inequality and Lemma~\ref{lem:tv-convergence} give
\begin{align*}
 d_{\TV}(\Law(F_n,Y_n),\Law(F,Y))
 &\le \PP\bigl((F_n,Y_n)\ne(\widetilde F_n,\widetilde Y_n)\bigr)\\
 &\quad+d_{\TV}(\Law(\widetilde F_n,\widetilde Y_n),\Law(F,Y))
 \longrightarrow0.
\end{align*}
Moreover,
\[
 \sup_n\EE\bigl[F_n^2\one_{\{|F_n|>M\}}\bigr]
 \le M^{-\eta}\sup_n\EE|F_n|^{2+\eta}\longrightarrow0.
\]
Lemma~\ref{lem:risk-continuity} proves \eqref{eq:localized-risk}.
\end{proof}

\subsection{Fixed designs}

We first keep the design fixed and identify its weak-disorder loss.  For
$c\in\R$ and $f\in V_K$, injectivity of $G^{1/2}$ gives
\[
 cv_0+v_f=0
 \Longrightarrow
 c\psi^2+\rho_f=0
 \Longrightarrow
 c=0,\quad f=0,
\]
because $\int_D\psi^2\dd x=1$ and $\int_D\rho_f\dd x=0$.  Likewise,
\[
 z=cv_0+v_f
 \Longleftrightarrow
 Lg=c\psi^2+\rho_f
 \Longleftrightarrow
 c=0,\quad f=f_g.
\]
Hence
\begin{equation}
 \dim U(V)=m+1,
 \qquad
 z\in U(V)\Longleftrightarrow f_g\in V.
 \label{eq:matching-equivalence}
\end{equation}
The map
\[
 \R\oplus V_K\ni(c,f)\longmapsto cv_0+v_f\in U_K,
 \qquad
 U_K:=\spanop\bigl(v_0,\{v_f:f\in V_K\}\bigr),
\]
is therefore an isomorphism.

On a Grassmann chart choose a continuous basis
$q_1(V),\ldots,q_m(V)$ and set
\[
 v_0(V)=v_0,
 \qquad
 v_j(V)=v_{q_j(V)},
 \qquad
 \Sigma(V)=\bigl(\ip{v_i(V)}{v_j(V)}\bigr)_{i,j=0}^{m}.
\]
On compact subsets of the chart, $\Sigma(V)$ is uniformly positive
definite.  Define
\[
 e_i(V)=\sum_{j=0}^{m}(\Sigma(V)^{-1/2})_{ij}v_j(V),
 \qquad 0\le i\le m,
\]
so that $(e_i(V))_{i=0}^m$ is orthonormal, and put
\begin{align*}
 X_V&=(W(e_0(V)),\ldots,W(e_m(V))),\\
 \xi_V&=\Sigma(V)^{-1/2}
      (\xi_0,\xi_{q_1(V)},\ldots,\xi_{q_m(V)})^{\mathsf T},\\
 Y_{\gamma,V}&=\Sigma(V)^{-1/2}
 (Y_{\gamma,0},Y_{\gamma,q_1(V)},\ldots,Y_{\gamma,q_m(V)})^{\mathsf T}.
\end{align*}
The vector $Y_{\gamma,V}$ generates the same sigma-algebra as
$S_{\gamma,V}$.  Let $\widetilde Y_{\gamma,V}$ denote the vector obtained
from the localized readouts.  Lemma~\ref{lem:observation-expansion} gives,
uniformly on compact chart subsets,
\begin{equation}
 \begin{aligned}
 Y_{\gamma,V}&=X_V+\gamma\xi_V+O_{L^p}(\gamma^2),\\
 \widetilde Y_{\gamma,V}
   &=X_V+\gamma\xi_V+O_{\Dsob^{r,p}}(\gamma^2),\\
 \PP(Y_{\gamma,V}\ne\widetilde Y_{\gamma,V})&=O(\gamma^M)
 \end{aligned}
 \label{eq:whitened-expansion}
\end{equation}
for every prescribed finite $r,p,M$ and all sufficiently small $\gamma$.

\begin{lemma}
\label{lem:fixed-design-limit}
If $V\in\Gr(m,V_K)$ is fixed and $f_g\in V$, then
\begin{equation*}
 \gamma^{-2}\ell_\gamma(V;g)\longrightarrow\kappa_g(U(V)).
\end{equation*}
\end{lemma}

\begin{proof}
Because $Y_{\gamma,f_g}$ is observed,
\[
 F_\gamma:=\frac{Y_{\gamma,f_g}-Z}{\gamma}
 \quad\Longrightarrow\quad
 \gamma^{-2}\ell_\gamma(V;g)=\Risk(F_\gamma\mid Y_{\gamma,V}).
\]
For
$\widetilde F_\gamma=(\widetilde Y_{\gamma,f_g}-Z)/\gamma$,
Lemma~\ref{lem:observation-expansion} gives, for suitable finite $p$ and some
$\eta>0$,
\[
 \widetilde F_\gamma
 =\xi_{f_g}+O_{\Dsob^{3,p}}(\gamma),
 \qquad
 F_\gamma=\xi_{f_g}+O_{L^{2+\eta}}(\gamma).
\]
Lemma~\ref{lem:error-tail} and Lemma~\ref{lem:chaos-nondegeneracy} give
nondegeneracy of $(\xi_{f_g},X_V)$.  Hence
Lemma~\ref{lem:localized-risk} and \eqref{eq:gaussian-conditioning} yield
\begin{align*}
 \lim_{\gamma\downarrow0}\gamma^{-2}\ell_\gamma(V;g)
 &=\Risk(I_2(K_g)\mid W(U(V)))\\
 &=2\norm{K_g-P_{U(V)}K_gP_{U(V)}}_{\HS}^2.
\end{align*}
Since \(K\mapsto P_{U(V)}KP_{U(V)}\) is the Hilbert--Schmidt orthogonal
projection onto operators supported on \(U(V)\),
\[
 \norm{K_g-P_{U(V)}K_gP_{U(V)}}_{\HS}^2
 =\norm{K_g}_{\HS}^2-
   \norm{P_{U(V)}K_gP_{U(V)}}_{\HS}^2.
\]
Thus the limit equals \(\kappa_g(U(V))\).
\end{proof}

\subsection{Quadratic correction of the score}
\label{subsec:score-correction}

The same second-order coefficient also yields an explicit quadratic correction
to the conditional score.  Fix a matched design $V$ and use the orthonormal
basis $e_0(V),\ldots,e_m(V)$ from the fixed-design construction.  Define the
symmetric matrix and its quadratic polynomial by
\begin{align*}
 M_V&=\bigl(\ip{e_i(V)}{K_ge_j(V)}\bigr)_{i,j=0}^{m},\\
 c_V(y)&=y^{\mathsf T}M_Vy-\Tr M_V+d_{f_g},\\
 \widehat Z_{\gamma,V}
   &=Y_{\gamma,f_g}-\gamma c_V(Y_{\gamma,V}).
\end{align*}
By \eqref{eq:whitened-expansion},
\[
 \norm{c_V(Y_{\gamma,V})-c_V(X_V)}_2
 \le \norm{M_V}\,
      \norm{Y_{\gamma,V}-X_V}_4
      (\norm{Y_{\gamma,V}}_4+\norm{X_V}_4)
 =O(\gamma).
\]
Moreover, \eqref{eq:gaussian-conditioning} gives
\[
 c_V(X_V)
 =\EE[\xi_{f_g}\mid W(U(V))]
 =I_2(P_{U(V)}K_gP_{U(V)})+d_{f_g}.
\]
Therefore
\begin{align*}
 \gamma^{-1}(Z-\widehat Z_{\gamma,V})
 &\longrightarrow
 -I_2(K_g-P_{U(V)}K_gP_{U(V)})
       &&\text{in }L^2,\\
 \EE|Z-\widehat Z_{\gamma,V}|^2
 &=\gamma^2\kappa_g(U(V))+o(\gamma^2).
\end{align*}
Lemma~\ref{lem:fixed-design-limit} and orthogonality of conditional
projection give
\[
 \norm{\EE[Z\mid S_{\gamma,V}]-\widehat Z_{\gamma,V}}_2^2
 =\EE|Z-\widehat Z_{\gamma,V}|^2-\ell_\gamma(V;g)
 =o(\gamma^2).
\]

\subsection{Designs depending on the disorder parameter}

We now return to the optimization problem and allow the design to vary with
$\gamma$.  Fix $K$ with $f_g\in V_K$, and fix any norm
$\norm{\cdot}_{V_K}$ on $V_K$.

\begin{lemma}
\label{lem:moving-matching-rate}
Let $\gamma_n\downarrow0$ and $V_n\to V_*$ in $\Gr(m,V_K)$.  Set
\[
 U_n:=U(V_n),
 \qquad
 P_n:=P_{U_n},
 \qquad
 u_n:=z-P_nz.
\]
If
\[
 \ell_{\gamma_n}(V_n;g)=O(\gamma_n^2),
\]
then
\[
 f_g\in V_*,
 \qquad
 \norm{u_n}=O(\gamma_n).
\]
\end{lemma}

\begin{proof}
If $z\notin U(V_*)$, then the Gaussian vector $(Z,X_{V_*})$ is
nondegenerate.  Hence \eqref{eq:whitened-expansion} and
Lemma~\ref{lem:localized-risk} give
\[
 \ell_{\gamma_n}(V_n;g)
 \longrightarrow\norm{z-P_{U(V_*)}z}^2>0,
\]
a contradiction.  Hence $f_g\in V_*$ by
\eqref{eq:matching-equivalence}.

Write uniquely
\begin{equation*}
 P_nz=\alpha_nv_0+v_{q_n},
 \qquad q_n\in V_n.
\end{equation*}
Since $P_nz-z=-u_n$ and the inverse of
$(c,f)\mapsto cv_0+v_f$ on $\R\oplus V_K$ is bounded,
\begin{equation}
 |\alpha_n|+\norm{q_n-f_g}_{V_K}\le C_K\norm{u_n}.
 \label{eq:moving-source-mismatch}
\end{equation}
Define the observable linear combination and its localized version by
\begin{equation*}
 \begin{aligned}
 Z_n^{\rm lin}&:=\alpha_nY_{\gamma_n,0}+Y_{\gamma_n,q_n},\\
 \widetilde Z_n^{\rm lin}&:=
 \alpha_n\widetilde Y_{\gamma_n,0}+\widetilde Y_{\gamma_n,q_n}.
 \end{aligned}
\end{equation*}
Linearity of the second-order coefficients and
\eqref{eq:moving-source-mismatch} give, for suitable finite $p$ and some
$\eta>0$,
\begin{equation}
 \begin{aligned}
 &\norm{(Z-Z_n^{\rm lin})-W(u_n)+\gamma_n\xi_{f_g}}_{L^{2+\eta}}\\
 &\quad+
 \norm{(Z-\widetilde Z_n^{\rm lin})-W(u_n)+\gamma_n\xi_{f_g}}_{3,p}
 \le C(\gamma_n\norm{u_n}+\gamma_n^2).
 \end{aligned}
 \label{eq:moving-residual-estimate}
\end{equation}
If $\norm{u_n}/\gamma_n\to\infty$ along a subsequence, then,
because $u_n\in U_K$, a further subsequence satisfies
\[
 \frac{u_n}{\norm{u_n}}\longrightarrow e,
 \qquad
 \norm{e}=1.
\]
The chart construction gives
\[
 P_n\longrightarrow P_{U(V_*)}\quad\text{in operator norm},
 \qquad P_nu_n=0,
\]
so $e\in U(V_*)^\perp$.
Divide \eqref{eq:moving-residual-estimate} by $\norm{u_n}$.  The actual target
converges to $W(e)$ in $L^{2+\eta}$ and the localized target in
$\Dsob^{3,p}$.  Since $W(e)$ is independent of $W(U(V_*))$ and
$Z_n^{\rm lin}$ is observable from $Y_{\gamma_n,V_n}$,
\[
 \frac{\ell_{\gamma_n}(V_n;g)}{\norm{u_n}^2}
 =\Risk\!\left(
   \frac{Z-Z_n^{\rm lin}}{\norm{u_n}}\,\middle|\,
   Y_{\gamma_n,V_n}
  \right)\longrightarrow1
\]
by Lemma~\ref{lem:localized-risk}, contradicting
$\ell_{\gamma_n}(V_n;g)=O(\gamma_n^2)$.  Thus
$\norm{u_n}=O(\gamma_n)$.
\end{proof}

\subsection{Enlarging the probe dictionary}

After the fixed-\(K\) asymptotic, only a deterministic compression problem remains.  We now remove the dictionary constraint and read the large-budget behavior from the tail of \(K_g\).

\begin{lemma}
\label{lem:dense-response-directions}
\begin{equation}
 \overline{\{v_f:f\in\mathcal P\}}^{\,H}=H,
 \qquad
 \overline{\bigcup_{K\ge1}\{v_f:f\in V_K\}}^{\,H}=H.
 \end{equation}
\end{lemma}

\begin{proof}
Suppose $u\perp v_f$ for every $f\in\mathcal P$ and put
$v=G^{1/2}u\in H_0^1(D)$.  By
\eqref{eq:linearization-factorization},
\[
 0=\ip{v}{\rho_f}
 =\ip{\psi v}{(I+2\lambda_0R_0)(f\psi)}.
\]
The set $\{f\psi:f\in\mathcal P\}$ is dense in $\psi^\perp$.  Indeed, for
$w\in\psi^\perp$, choose $w_j\in C_c^\infty(D)$ with $w_j\to w$ in $H$,
put $a_j=w_j/\psi$, and set
\[
 f^{(j)}=a_j-\ip{w_j}{\psi}\in\mathcal P.
\]
Then
\[
 f^{(j)}\psi=w_j-\ip{w_j}{\psi}\psi\longrightarrow w
 \quad\text{in }H.
\]
Since $I+2\lambda_0R_0$ is invertible on $\psi^\perp$,
\[
 Q(\psi v)=0
 \quad\Longrightarrow\quad
 v=c\quad\text{a.e.}
 \quad\Longrightarrow\quad
 c=0,
\]
the last implication following from the zero trace of $v$.  Hence $u=0$,
which proves the first density statement.

Finally, \eqref{eq:linearization-factorization} gives
\[
 \norm{v_f-v_{\widetilde f}}
 \le C\norm{f-\widetilde f}_\infty.
\]
The uniform density of $\bigcup_KV_K$ in the centered continuous functions
then gives the second density statement.
\end{proof}

Let
\[
 U_0:=\spanop(v_0,z).
\]
The two vectors are linearly independent because their source measures have
total masses $1$ and $0$, respectively, and $z\ne0$.

\begin{lemma}
\label{lem:compression-bounds}
For every $m\ge1$,
\begin{equation}
 \kappa_m(g)
 =\inf_{\substack{U\supset U_0\\\dim U=m+1}}
 2\bigl(\norm{K_g}_{\HS}^2-\norm{P_UK_gP_U}_{\HS}^2\bigr),
 \label{eq:unrestricted-coefficient}
\end{equation}
and
\begin{equation}
 2\sum_{j>m+1}s_j(K_g)^2
 \le\kappa_m(g)
 \le2\sum_{j>m-1}s_j(K_g)^2.
 \label{eq:compression-tail-bounds}
\end{equation}
\end{lemma}

\begin{proof}
Every matched design gives an $(m+1)$-dimensional subspace containing
$U_0$, so the right-hand side of \eqref{eq:unrestricted-coefficient} is a
lower bound for $\kappa_m(g)$.  Conversely, let $U\supset U_0$ with $\dim U=m+1$ and choose
$w_1,\ldots,w_{m-1}$ such that
\[
 U=\spanop(v_0,z,w_1,\ldots,w_{m-1}).
\]
By Lemma~\ref{lem:dense-response-directions}, for each $i$ choose
$f_i^{(n)}\in\bigcup_KV_K$ with $v_{f_i^{(n)}}\to w_i$.  Since the
dictionaries are nested, for each $n$ there is $K_n$ such that
\[
 V^{(n)}:=\spanop(f_g,f_1^{(n)},\ldots,f_{m-1}^{(n)})
 \subset V_{K_n}.
\]
For all sufficiently large $n$, $\dim V^{(n)}=m$, and
\[
 U^{(n)}:=U(V^{(n)})
 =\spanop(v_0,z,v_{f_1^{(n)}},\ldots,v_{f_{m-1}^{(n)}})
\]
satisfies
\[
 \norm{P_{U^{(n)}}-P_U}\longrightarrow0.
\]
Hence
\[
 \norm{P_{U^{(n)}}K_gP_{U^{(n)}}-P_UK_gP_U}_{\HS}
 \le2\norm{P_{U^{(n)}}-P_U}\norm{K_g}_{\HS}\longrightarrow0.
\]
By the definition of $\kappa_{m,K_n}(g)$ and the monotonicity in $K$,
\[
 \kappa_m(g)
 \le\lim_{n\to\infty}
 2\bigl(\norm{K_g}_{\HS}^2
   -\norm{P_{U^{(n)}}K_gP_{U^{(n)}}}_{\HS}^2\bigr)
 =2\bigl(\norm{K_g}_{\HS}^2-\norm{P_UK_gP_U}_{\HS}^2\bigr).
\]
Taking the infimum over such $U$ proves \eqref{eq:unrestricted-coefficient}.

For $d=\dim U$,
\[
 \norm{P_UK_gP_U}_{\HS}^2
 \le\Tr(P_UK_g^2P_U)
 \le\sum_{j\le d}s_j(K_g)^2,
\]
which gives the lower bound in \eqref{eq:compression-tail-bounds} with
$d=m+1$.  For the upper bound, choose $U$ to contain $U_0$ and eigenvectors
corresponding to the $m-1$ largest absolute eigenvalues of $K_g$, adding
directions if necessary.  Then
\[
 \norm{P_UK_gP_U}_{\HS}^2
 \ge\sum_{j\le m-1}s_j(K_g)^2,
\]
which proves the upper bound.
\end{proof}

\subsection{Final proof of Theorem~\ref{thm:optimal}}

\begin{proof}
Fix $m,K$ with $f_g\in V_K$.  The compact set
\[
 \{V\in\Gr(m,V_K):f_g\in V\}
\]
is nonempty, and $V\mapsto\kappa_g(U(V))$ is continuous.  Since $K_g$ has
infinite rank,
\[
 \kappa_g(U(V))
 =2\norm{K_g-P_{U(V)}K_gP_{U(V)}}_{\HS}^2>0
\]
for every matched $V$.  Hence
\[
 0<\kappa_{m,K}(g)<\infty.
\]
If $V^\circ$ minimizes the matched compression problem, then
Lemma~\ref{lem:fixed-design-limit} gives
\begin{equation}
 \limsup_{\gamma\downarrow0}
 \frac{\Delta_{\gamma,m,K}(g)}{\gamma^2}
 \le\kappa_{m,K}(g).
 \label{eq:optimal-upper-bound}
\end{equation}

For the lower bound, choose $\gamma_n\downarrow0$ realizing the lower limit
and minimizers $V_n$ of $\ell_{\gamma_n}(\cdot;g)$.  By
\eqref{eq:optimal-upper-bound},
\[
 \ell_{\gamma_n}(V_n;g)=O(\gamma_n^2).
\]
By compactness of \(\Gr(m,V_K)\), after passing to a subsequence,
\[
 V_n\longrightarrow V_*.
\]
Set
\[
 U_n:=U(V_n),
 \qquad P_n:=P_{U_n},
 \qquad u_n:=z-P_nz.
\]
Lemma~\ref{lem:moving-matching-rate} gives
\[
 f_g\in V_*,
 \qquad
 \norm{u_n}=O(\gamma_n).
\]
Since \(u_n\in U_K\) and \(U_K\) is finite-dimensional, a further
subsequence satisfies
\[
 \frac{u_n}{\gamma_n}\longrightarrow v.
\]
With
\[
 U_*:=U(V_*),
 \qquad P_*:=P_{U_*},
\]
one has $P_n\to P_*$ and $P_nu_n=0$, hence
\[
 v\in U_*^\perp.
\]
Write
\[
 \begin{gathered}
 P_nz=\alpha_nv_0+v_{q_n},\qquad q_n\in V_n,\\
 Z_n^{\rm lin}=\alpha_nY_{\gamma_n,0}+Y_{\gamma_n,q_n},\qquad
 \widetilde Z_n^{\rm lin}
 =\alpha_n\widetilde Y_{\gamma_n,0}+\widetilde Y_{\gamma_n,q_n}.
 \end{gathered}
\]
Since the inverse source map on $U_K$ is bounded,
\[
 |\alpha_n|+\norm{q_n-f_g}_{V_K}\le C_K\norm{u_n}.
\]
Set $\xi_n=\alpha_n\xi_0+\xi_{q_n}$.  By linearity of the second-order
coefficients and the preceding bound,
\[
 \norm{\xi_n-\xi_{f_g}}_{3,p}\le C_{K,p}\norm{u_n}.
\]
The readout expansions therefore give, for some $\eta>0$,
\begin{align*}
 Z_n^{\rm lin}
 &=W(P_nz)+\gamma_n\xi_n+O_{L^{2+\eta}}(\gamma_n^2),\\
 \widetilde Z_n^{\rm lin}
 &=W(P_nz)+\gamma_n\xi_n+O_{\Dsob^{3,p}}(\gamma_n^2),
\end{align*}
whence
\begin{equation}
 \begin{aligned}
 &\norm{(Z-Z_n^{\rm lin})-W(u_n)+\gamma_n\xi_{f_g}}_{L^{2+\eta}}\\
 &\quad+
 \norm{(Z-\widetilde Z_n^{\rm lin})-W(u_n)+\gamma_n\xi_{f_g}}_{3,p}
 \le C(\gamma_n\norm{u_n}+\gamma_n^2).
 \end{aligned}
 \label{eq:final-residual-estimate}
\end{equation}
Dividing \eqref{eq:final-residual-estimate} by $\gamma_n$ yields
\begin{align*}
 \frac{Z-Z_n^{\rm lin}}{\gamma_n}
 &\longrightarrow W(v)-I_2(K_g)-d_{f_g}
       &&\text{in }L^{2+\eta},\\
 \frac{Z-\widetilde Z_n^{\rm lin}}{\gamma_n}
 &\longrightarrow W(v)-I_2(K_g)-d_{f_g}
       &&\text{in }\Dsob^{3,p}.
\end{align*}
The limiting joint vector with $X_{V_*}$ is nondegenerate by
Lemma~\ref{lem:chaos-nondegeneracy}.  The actual and localized pairs differ
only on the common localization event, and
\eqref{eq:final-residual-estimate} gives uniformly bounded
\((2+\eta)\)-moments for the actual normalized targets.  Since
$Z_n^{\rm lin}$ is observable from $Y_{\gamma_n,V_n}$,
\[
 \frac{\ell_{\gamma_n}(V_n;g)}{\gamma_n^2}
 =\Risk\!\left(
   \frac{Z-Z_n^{\rm lin}}{\gamma_n}\,\middle|\,
   Y_{\gamma_n,V_n}
  \right).
\]
Lemma~\ref{lem:localized-risk} therefore gives
\begin{align*}
 \lim_n\frac{\ell_{\gamma_n}(V_n;g)}{\gamma_n^2}
 &=\Risk(W(v)-I_2(K_g)\mid W(U_*))\\
 &=\norm{v}^2+
 2\norm{K_g-P_*K_gP_*}_{\HS}^2\\
 &=\norm{v}^2+\kappa_g(U_*)
 \ge\kappa_{m,K}(g).
\end{align*}
Together with \eqref{eq:optimal-upper-bound}, this proves
\eqref{eq:main-optimal}.

Suppose now that
\[
 \ell_\gamma(V_\gamma;g)-\Delta_{\gamma,m,K}(g)=o(\gamma^2).
\]
By \eqref{eq:main-optimal},
\[
 \ell_\gamma(V_\gamma;g)
 =\gamma^2\kappa_{m,K}(g)+o(\gamma^2).
\]
Fix any sequence \(\gamma_n\downarrow0\) and pass to a subsequence such that
\(V_{\gamma_n}\to V_*\).  Lemma~\ref{lem:moving-matching-rate} gives
\[
 f_g\in V_*,
 \qquad
 u_n:=z-P_{U(V_{\gamma_n})}z=O(\gamma_n).
\]
After a further subsequence, \(u_n/\gamma_n\to v\).  The residual
calculation in \eqref{eq:final-residual-estimate} and
Lemma~\ref{lem:localized-risk} give
\[
 \lim_n\frac{\ell_{\gamma_n}(V_{\gamma_n};g)}{\gamma_n^2}
 =\norm{v}^2+\kappa_g(U(V_*)).
\]
The preceding near-optimal expansion therefore yields
\[
 \norm{v}^2+\kappa_g(U(V_*))=\kappa_{m,K}(g).
\]
Since \(V_*\) is matched and
\(\kappa_g(U(V_*))\ge\kappa_{m,K}(g)\),
\[
 v=0,
 \qquad
 \kappa_g(U(V_*))=\kappa_{m,K}(g).
\]
Thus every cluster point minimizes the matched compression problem.  If
\(\norm{z-P_{U(V_\gamma)}z}/\gamma\not\to0\), one could choose
\(\gamma_n\downarrow0\) for which this ratio is bounded away from zero;
the preceding compactness argument would then produce a nonzero limit \(v\),
a contradiction.  Hence
\[
 \norm{z-P_{U(V_\gamma)}z}=o(\gamma).
\]

Finally, with
\[
 c_g=\frac14\norm{Lg}_{L^1(D)},
\]
Lemma~\ref{lem:error-tail} gives
\[
 \sum_{j>d}s_j(K_g)^2\sim\frac{c_g^2}{d}.
\]
Lemma~\ref{lem:compression-bounds}, through \eqref{eq:compression-tail-bounds}, yields
\[
 \lim_{m\to\infty}m\kappa_m(g)
 =2c_g^2
 =\frac18\norm{Lg}_{L^1(D)}^2,
\]
which is \eqref{eq:main-budget}.  Since
$\kappa_{m,K}(g)\downarrow\kappa_m(g)$ as $K\to\infty$,
\eqref{eq:main-ordered-limit} follows in the stated order.
\end{proof}

\end{document}